\documentclass{amsart}

\usepackage{comment}
\usepackage{color}
\usepackage{graphicx}
\usepackage{verbatim}
\usepackage{tensor}
\newtheorem{thm}{Theorem}[section]
\newtheorem{cor}[thm]{Corollary}
\newtheorem{lem}[thm]{Lemma}
\newtheorem{prop}[thm]{Proposition}

\newtheorem{ques}{Question}
\theoremstyle{definition}

\theoremstyle{remark}
\newtheorem{rem}[thm]{Remark}

\numberwithin{equation}{section}

\newcommand{\set}[1]{\left\{#1\right\}}
\definecolor{red}{rgb}{0.7,0,0}

\newcommand{\Real}{\mathbb R}

\newcommand{\func}[1]{\ensuremath{\operatorname{#1} } }

\newcommand{\Div}[0]{\func{div}}

\newcommand{\re}[0]{\func{Re}}

\newcommand{\tr}[0]{\func{tr}}

\newcommand{\xX}[0]{\mathbf{x}}
\newcommand{\eE}[0]{\mathbf{e}}
\newcommand{\FF}[0]{\mathbf{F}}

\newcommand{\n}{\mathbf{n}}
\newcommand{\x}{\mathbf{x}}
\newcommand{\R}{\mathbb{R}}
\newcommand{\I}{i}

\newcommand{\wW}[0]{\mathbf{w}}

\newcommand{\nN}[0]{\mathbf{n}}

\newcommand{\gnabla}[0]{\tensor[^g]{\nabla}{}}
\newcommand{\rgnabla}{\tensor[^g]{\mathring{\nabla}}{}}

\newcommand{\Ric}[0]{\func{Ric}}
\renewcommand{\d}{\mathrm{d}}

\newcommand{\vol}{\mu_g}

\title[Characterizing Minimal Surfaces]{Characterizing classical minimal surfaces via the entropy differential}
\author[J. Bernstein \and T. Mettler]{Jacob Bernstein \and Thomas Mettler}

\thanks{The first author was partially supported by the NSF grants DMS-0902721 and DMS-1307953 and by the EPSRC Programme Grant 
``Singularities of Geometric Partial Differential Equations'' number EP/K00865X/1. The second author was partially supported by the 
Mathematical Sciences Research Institute in Berkeley and by the  Forschungsinstitut f\"ur Mathematik at ETH Z\"urich.}

\subjclass[2010]{53A10, 70S10}

\keywords{Minimal surfaces, conservation laws, Schwarzian derivative}

\date{August 11, 2016}

\address{Department of Mathematics, Johns Hopkins University, Baltimore, MD 21218, USA}
\email{bernstein@math.jhu.edu}
\address{Institut f\"ur Mathematik, Goethe-Universit\"at Frankfurt, 60325 Frankfurt am Main, Germany}
\email{mettler@math.uni-frankfurt.de}

\begin{document}

\begin{abstract}
We introduce on any smooth oriented minimal surface in Euclidean $3$-space a meromorphic quadratic differential, $P$, which we call the \emph{entropy differential}. This differential arises naturally in a number of different contexts.  Of particular interest is the realization of its real part as a conservation law for a natural geometric functional -- which is, essentially, the entropy of the Gauss curvature.  We characterize several classical surfaces -- including Enneper's surface, the catenoid and the helicoid -- in terms of $P$.  As an application, we prove a novel curvature estimate for embedded minimal surfaces with small entropy differential and an associated compactness theorem.
\end{abstract}
\maketitle

\section{Introduction}
Let $\Sigma \subset \R^3$ be a smooth, oriented minimal surface. In this paper, we introduce a meromorphic quadratic differential $P$ on $\Sigma$, which we call the \emph{entropy differential}. We use $P$ to characterize several classical surfaces -- including Enneper's surface, the catenoid and the helicoid. In particular, subsets of Enneper's surface are the only minimal surfaces on which $P$ vanishes -- a fact which we use to prove a novel curvature estimate for embedded minimal surfaces with small entropy differential and an associated compactness result. 

The differential $P$ arises naturally in a number of different contexts.  Of particular interest is the realization of $T=\re P$, which we call the \emph{entropy form}, as a conservation law for the diffeomorphism invariant functional
\begin{equation*}
 \mathcal{E}[g]=\int_{\Sigma} K_g \log K_g \vol.
\end{equation*}
 This functional, which is a type of entropy for the curvature, has been previously considered by R. Hamilton in the context of the Ricci flow on surfaces ~\cite{Hamilton2Sphere}.  In particular, we show that if $g$ is a minimal surface metric (i.e. the metric induced by a smooth minimal immersion) for which $K_g\neq 0$, then the metric $\hat{g}=(-K_g)^{3/4} g$ is a critical point of $\mathcal{E}$ with respect to compactly supported conformal deformations.  The crucial fact used here is the observation -- due to Ricci~\cite{Ricci} -- that such minimal surface metrics satisfy the so-called \emph{Ricci condition}:
\begin{equation*}
 \Delta_g \log |K_g|=4K_g.
\end{equation*}
The differential $P$ also arises as a certain geometric Schwarzian derivative of the Gauss map -- a point of view which has antecedents in~\cite{DurenEtAl,KusnerConversation} -- and which we will study more thoroughly in a forthcoming paper~\cite{jacobthomasdegree}.  

A key observation of the present paper is that, modulo rigid motions, a minimal surface is determined, up to a three-parameter 
family, by its Hopf differential $Q$ and its entropy differential $P$.  This allows one to characterize several classical minimal surfaces 
in terms of simple relationships between the Hopf and entropy differentials:
\setcounter{section}{4}
\begin{thm}\label{IntroMainCharacterizationThm}
 Let $\Sigma$ be a smooth oriented non-flat minimal surface in $\Real^3$ with entropy differential $P$. We have:
\begin{enumerate}
 \item \label{IntroEnneperCase} If $P\equiv 0$, then up to a rigid motion and homothety, $\Sigma$ is contained in Enneper's surface;
 \item \label{IntroCatCase} If $\lambda \neq 0$ and  $P\equiv \lambda Q$, then,  up to a rigid motion and homothety, $\Sigma$ is contained in a surface $C\in\mathcal{C}$.  If $\Sigma$ is properly embedded, then it is the catenoid;
\item \label{IntroHelCase} If $\lambda \neq 0$ and $P \equiv  \I \lambda Q$, then, up to a rigid motion and homothety, $\Sigma$ is contained  in a surface $H\in \mathcal{H}$.  If $\Sigma$ is properly embedded, then it is the helicoid.
\end{enumerate}
\end{thm}
The families $\mathcal{C}$ and $\mathcal{H}$ are, respectively, the \textit{deformed catenoids} and \textit{deformed helicoids}.  These are one parameter families of surfaces containing, respectively, the catenoid and the helicoid -- their geometry is discussed thoroughly in Section \ref{CharSec}.  

A consequence of Item \eqref{IntroEnneperCase} of Theorem \ref{IntroMainCharacterizationThm} are a family of novel curvature estimates for 
embedded 
minimal surfaces.  Namely, we introduce a certain family of scale invariant quantities which measure the size of the  entropy form and use 
standard blow-up arguments to derive curvature bounds for embedded surfaces for which these 
quantities are small.  Specifically, for a constant $\alpha>0$ and smooth minimal surface $\Sigma$ with entropy form $T$, we define:
\begin{equation*}
||T||_{\Sigma, \alpha}:=2^{\frac{1}{2(1+\alpha)}}\int_{\Sigma} |T|^{\frac{1}{1+\alpha}}_g |K_g|^{\frac{\alpha}{\alpha+1}}  \vol.
\end{equation*}
We justify this family by noting that, on the one hand they are scale invariant and, on the other, the ``endpoints''  are 
very natural.  Indeed, 
$$
\lim_{\alpha\to \infty} ||T||_{\Sigma, \alpha}=\int_{\Sigma} |K_{g}| \vol,
$$ i.e., one endpoint is the total Gauss curvature, a well studied quantity in minimal surface theory.
While,
$$
\lim_{\alpha\to 0} ||T||_{\Sigma, \alpha}=\sqrt{2} \int_{\Sigma} |T|_g \vol,
$$
i.e., the other endpoint is the $L^1$ norm of the entropy form which is invariant under the standard 
action of $\rm PSL(2,\mathbb{C})$ on the Gauss map of  $\Sigma$, see \cite{jacobthomasdegree}.  We will not 
deal directly with this quantity due to the fact that the presence of umbilic points tends to make it infinite. 

\setcounter{section}{5}
\setcounter{thm}{2}

As these quantities are scale invariant, standard blow-up arguments give the following curvature estimates:
\begin{thm}\label{IntroCurvEstThm}
 Given $\alpha>0$, there exist constants  $\epsilon=\epsilon(\alpha)>0$ and $C=C(\alpha)>0$ so that: if $\Sigma$ is a properly embedded 
minimal surface in $B_{2R}$ and 
$$||T||_{\Sigma,\alpha}<\epsilon,$$
 then
\begin{equation*}
 R^2 \sup_{B_{R}\cap \Sigma} |A|^2 \leq C^2.
\end{equation*}
\end{thm}
In Section \ref{GapSec}, we address the question of the best possible $\epsilon$ and some partial results are obtained.
In particular, we obtain the following identity:
\setcounter{thm}{9}
\begin{cor}
 Let $\Sigma$ be a non-flat properly immersed minimal surface in $\Real^3$ of finite total Gauss curvature with genus $g$ and $e$ 
embedded ends, 
 then
 \begin{align*}
 \lim_{\alpha\to 0} \alpha ||{T}||_{\Sigma, \alpha} &=\frac{\pi}{4}\left( 8+12 g+10(e-2)+ \sum_{p\in \mathcal{E}\bigcup 
\mathcal{U}}  
\frac{n(p)}{n(p)+1} \right).
 \end{align*}
 Here $\mathcal{E}$ is the set of ends and $n(p)\geq 0$ for $p\in \mathcal{E}$ is the order of branching of the 
end, i.e., the order of branching of the extension of the Gauss map to $p$, while $\mathcal{U}$ is the 
set of umbilic points and $n(p)\geq 1$ is the order of the umbilic point for $p\in \mathcal{U}$.
\end{cor}
This suggests that as $\alpha\to 0$
$$
\epsilon(\alpha)=\frac{2\pi}{\alpha} +o\left( \frac{1}{\alpha}\right).
$$
which would be sharp on the catenoid.
Using standard techniques, we observe that our curvature estimate gives a corresponding compactness results which we record in  Theorem 
\ref{CpctThm}.
We conclude the paper with Appendix \ref{RicciFlowSec}, wherein the entropy form is used to make a connection between minimal surfaces in $\Real^3$ and gradient Ricci soliton metrics on surfaces.

\setcounter{section}{1}

\subsection*{Acknowledgements}
The authors would like to thank Rob Kusner and Daniel Fox for several stimulating discussions regarding the topics of this paper. The authors are also grateful to the anonymous referee for carefully reading the article and many useful suggestions. 

\section{A Geometric Entropy Functional for Surfaces}

\subsection{Definitions}
We assume $\R^3$ to be equipped with the standard Euclidean metric $g_E$ and orientation. Let $B_{R}(p)$ be the open Euclidean ball in $\R^3$ with radius $R>0$ and center $p$. If $p$ is omitted then the ball is assumed to be centered at the origin in $\R^3$. Let $M$ be an open orientable smooth $2$-manifold. For a smooth immersion $\x : M \to \R^3$ let $\Sigma=\x(M)$, we say $\Sigma$ is properly embedded if $\x$ is proper and injective. Moreover, let $g=\x^*g_E$ be the first fundamental form.  We write $\gnabla$ for the Levi-Civita connection, $K_g$ for the Gauss curvature, and $\vol$ for the area form of $g$. The integrable almost complex structure on $M$ induced by $g$ and the orientation will be denoted by $J$. Furthermore, for $A \in \Gamma(S^2(T^*M))$, we define
$
(JA)(X,Y)=A(JX,Y)
$
where $X,Y \in \Gamma(TM)$ are smooth vector fields on $M$. Here, as usual, $S^2(T^*M)$ denotes the second symmetric power of the cotangent bundle of $M$. In particular, the map $A\mapsto A+i JA$ embeds the space of symmetric trace-free $2$-forms on $M$ into the space of quadratic differentials on $M$. Furthermore, we use the standard fact that $A+\I JA$ is holomorphic if and only if $A$ is divergence-free.
 
Let $\n$ denote the orientation compatible Gauss map of $\x$ taking values in $\partial B_1$, the unit-sphere in $\R^3$ centered at $0$. The second fundamental form of $\x$ will be denoted by $A$ and its trace with respect to $g$, the \textit{mean curvature}, by $H$. A point $p\in M$ at which the eigenvalues of $A_p$ agree is called \emph{umbilic} and we define $\hat{M}\subset M$ to be the open submanifold of non-umbilic points. 

The pair $(g,A)$ satisfies the Gauss equation
\begin{equation*}
 |A|_g^2+2K_g =(\tr_g\! A)^2
\end{equation*}
and the Codazzi equations
\begin{equation*}
 \gnabla_X A (Y,Z)=\gnabla_Y A(X,Z)
\end{equation*}
where $X,Y,Z\in \Gamma(TN)$. Conversely, Bonnet's theorem states that if a pair $(g,A)$ on a simply connected surface $N$ satisfies the Gauss - and Codazzi equations, then there exists an immersion $\x : N \to \R^3$ -- unique up to composition with a rigid motion of $\R^3$ -- whose first and second fundamental form are $g$ and $A$. For this reason we refer to the triple $(M,g,A)$ as \emph{geometric data} of $\x$. 

\subsection{The Ricci condition} We suppose from now on that $\x : M \to \R^3$ is minimal.  That is $H\equiv 0$. The Gauss equations imply that $K_g\leq 0$ and that $K_g$ is negative on $\hat{M}$. It follows from the Codazzi equations that the second fundamental form is divergence free with respect to $g$. This yields \textit{Simons' identity}
\begin{equation*}
 \Delta_g A =-|A|_g^2 A
\end{equation*}
where here $\Delta_g$ is the rough Laplacian. From 
$$
4 |A|_g^2 |\gnabla A|^2_g=|\gnabla |A|_g^2|^2,
$$
Simons' identity, and the Gauss equations we obtain that on $\hat{M}$ the following \textit{Ricci condition}
\begin{equation} \label{RicIdent}
 \Delta_g \log |{K}_g| = 4 K_g
\end{equation}
holds. Abbreviate 
\begin{equation}\label{eq:ugdef}
u_g=-\frac{1}{4}\log|K_g|,
\end{equation}
then the Ricci condition becomes
$$
\Delta_gu_g=e^{-4u_g}. 
$$

Conversely, Ricci~\cite{Ricci} showed that if $g$ is a Riemannian metric of strictly negative Gauss curvature $K_g$ on a simply connected 
$2$-manifold $N$ satisfying the Ricci condition, then there exists a minimal immersion $\x : N \to \R^3$ with $\x^*g_E=g$. A proof of this 
fact using modern language may be found in~\cite{MR655419}. 

\subsection{The entropy functional}
We will study a certain natural functional $\mathcal{E}$ defined  on the space $\mathcal{M}^+(M)$ of smooth Riemannian metrics on $M$ which have positive Gauss curvature. 
Define for $g\in \mathcal{M}^+(M)$
\begin{equation*}
\mathcal{E}[g]= \int_M K_g \log K_g \vol
\end{equation*}
This functional has been applied to the study of Ricci flow on surfaces by Hamilton~\cite{Hamilton2Sphere} and Chow~\cite{Chow1991} -- in particular Hamilton observed that it is monotonically increasing along the Ricci flow on spheres with positive Gauss curvature (see Appendix~\ref{RicciFlowSec} of the present paper for additional connections to Ricci solitons).

We compute the Euler-Lagrange equations associated to $\mathcal{E}$.
To do so, let $f$ be a smooth symmetric $2$-form and write $F=\tr_g f$ for its trace with respect to $g$. Let $g_t=g+tf$, then (cf.~\cite[pg. 99]{ChowBook})
\begin{align*}
 \left.\frac{\partial}{\partial t}\right|_{t=0} K_{g_t} &= - \frac{1}{2} \Delta_g F + \frac{1}{2} \Div_g(\Div_g\! f)- \frac{1}{2} K_g F\\
&= - \frac{1}{4} \Delta_g F + \frac{1}{2} \Div_g\left(\Div_g\! \mathring{f}\right)- \frac{1}{2} K_g F
\end{align*}
where $\mathring{f}$ is the trace-free part of $f$.
Hence,
\begin{equation*}
\left.\frac{\partial}{\partial t}\right|_{t=0} K_{g_t} \log K_{g_t}  =  \left(-\frac{1}{4} \Delta_g F + \frac{1}{2} \Div_g\left(\Div_g\! \mathring{f}\right)- \frac{1}{2} K_g F\right) (\log K_{g_t}+1)
\end{equation*}
and so 
\begin{equation*}
 \delta_f\mathcal{E}[g]= \frac{1}{2}\int_{M}F K_g \log K_g+\left(\Div_g\left(\Div_g\! \mathring{f}\right)-\frac{1}{2} \Delta_g F - K_g F\right) (\log K_{g_t}+1)\vol.
\end{equation*}
If $f$ is compactly supported 
then using Green's formula and the divergence theorem, i.e., integrating by parts twice, yields 
\begin{equation*}
 \delta_f\mathcal{E} [g]= -\frac{1}{4}\int_{M} F \left(\Delta_g \log K_g +2  K_g   \right)  -2\langle f, \rgnabla^2 \log K_g\rangle_g \vol,
\end{equation*}
where $\rgnabla^2$ denotes the trace-free Hessian and $\langle a,b\rangle_g$ the natural bilinear pairing on elements $a,b\in \Gamma(S^2(T^*M))$ obtained via $g$. 

We will say that $g\in \mathcal{M}^+(M)$ is an $\mathcal{E}$-critical metric if $\mathcal{E}$ is stationary at $g$ with respect to compactly supported conformal deformations.
Hence,  $g$ is an $\mathcal{E}$-critical metric if and only if the Gauss curvature $K_g$ of $g$ satisfies
\begin{equation*}
 \Delta_g \log K_g = -2 K_g. 
\end{equation*}

As $\mathcal{E}$ is computed purely in terms of geometric quantities it is manifestly diffeomorphism invariant, that is
if $\phi: M \to M$ is a diffeomorphism we have
\begin{equation*}
 \mathcal{E}[\phi^* g] =\mathcal{E}[g].
\end{equation*}
By Noether's principle this invariance leads to a conservation law for $\mathcal{E}$-critical metrics.  Indeed, let $X$ be a compactly supported vector field on $M$ and $\phi_t$ the flow of $X$. We have that
\begin{equation*}
 \phi_t^* g=g+ t L_X g +o(t).
\end{equation*}
Recall, the Lie derivative, $L_X g$ can be computed as
\begin{equation*}
 (L_X g)(Y,Z)=g( \gnabla_Y X,Z) + g(\gnabla_Z X, Y).
\end{equation*}
where $Y,Z \in \Gamma(TM)$. By the diffeomorphism invariance we have at an $\mathcal{E}$-critical metric that
\begin{align*}
 0 &= \delta_{L_X g} \mathcal{E} [g] \\ &= \frac{1}{2} \int_{M} \langle L_X g, \rgnabla^2 \log K_g\rangle_g \vol\\
   &=\int_{M} \Div_g\! \left( (\rgnabla^2 \log K_g )(X, \cdot)\right) - \langle X, \Div_g\! \rgnabla^2 \log K_g \rangle_g \vol\\
   &=-\int_{M} \langle X, \Div_g\! \rgnabla^2 \log K_g \rangle_g \vol
\end{align*}
where we used that $X$ has compact support and the divergence theorem.
As $X$ is arbitrary,
\begin{equation}\label{divfree}
 \Div_g\! \rgnabla^2 \log K_g =0.
\end{equation}
In other words, the quantity
\begin{equation*}
V_g=\rgnabla^2 \log K_g  
\end{equation*}
 is a trace-free divergence free symmetric $2$-form, i.e.~a conservation law for the $\mathcal{E}$ functional.

\subsection{The entropy form}
Let $g$ be a smooth Riemannian metric and $\omega$ a smooth real-valued function on $M$. We note the following standard formula for the trace-free Hessian and the Laplacian operating on $u\in C^{\infty}(M)$
\begin{equation}\label{conformaltrafo}
\aligned
\Delta_{e^{2\omega}g}u&=e^{-2\omega}\Delta_gu,\\
^{e^{2\omega}g}\mathring{\nabla}^2&u=\rgnabla^2 u-\left(\d u \otimes \d \omega+\d \omega \otimes \d u-g(\gnabla u,\gnabla \omega)g\right).
\endaligned
\end{equation}
Also, the Gauss-curvature transforms under conformal change as
$$
K_{e^{2\omega}g}=e^{-2\omega}\left(K_g-\Delta_g \omega \right).  
$$
We let $\mathcal{R}^{\pm}_{C}$ denote the space of smooth positively (or negatively) curved Riemannian metrics on $M$ satisfying the \textit{generalized Ricci condition}
$$
\Delta_g\log|K_g|=CK_g
$$
for some real constant $C$. In particular, the $\mathcal{E}$-critical metrics are the elements of $\mathcal{R}^+_{-2}$. For $g \in \mathcal{R}^{\pm}_{C}$ and $\alpha \in \R$ let $g_{\alpha}=|K_g|^{2\alpha}g$, then $g_{\alpha}$ has Gauss-curvature
$$
K_{g_{\alpha}}=(1-C\alpha)|K_g|^{-2\alpha}K_g
$$
which, for $\alpha \neq \frac{1}{C}$, satisfies
\begin{equation*}
 \Delta_{g_{\alpha}} \log |K_{g_\alpha}| = \left(\frac{2\alpha -1}{\alpha-\frac{1}{C}}\right)K_{g_\alpha}=C_{\alpha}K_{g_{\alpha}},
\end{equation*}
where
$$
C_{\alpha}=\left(\frac{2\alpha -1}{\alpha-\frac{1}{C}}\right).
$$
It follows that for $\alpha>\frac{1}{C}$ the map $\varphi_{\alpha}$ sending $g$ to $g_{\alpha}$ satisfies
$$
\varphi_{\alpha} : \mathcal{R}^{\pm}_{C} \to \mathcal{R}^{\mp}_{C_{\alpha}}
$$
whereas for $\alpha < \frac{1}{C}$ 
$$
\varphi_{\alpha} : \mathcal{R}^{\pm}_{C} \to \mathcal{R}^{\pm}_{C_{\alpha}}.
$$
Note that the choice $\alpha=\frac{1}{C}$ maps the elements of $\mathcal{R}^{\pm}_{C}$ to flat metrics and the choice $\alpha=\frac{1}{2}$ (assuming $C\neq 2$) maps the elements of $\mathcal{R}^{\pm}_{C}$ to metrics of non-zero constant Gauss curvature.

Suppose $g \in \mathcal{R}^{-}_{4}$, then $\hat{g}=g_{3/8}$ is an $\mathcal{E}$-critical metric with Gauss curvature 
$$
K_{\hat{g}}=\frac{1}{2}|K_g|^{1/4}.
$$
It follows with~\eqref{eq:ugdef} and \eqref{divfree} that the symmetric trace-free \textit{entropy form} 
$$
\aligned
T&:=V_{\hat{g}}={}^{\hat{g}}\mathring{\nabla}^2\log K_{\hat{g}}={}^{\hat{g}}\mathring{\nabla}^2\left(\frac{1}{4}\log |K_g|\right)\\&=-{}^{\hat{g}}{\mathring{\nabla}}^2u_g=-\rgnabla^2u_g-3\left((\d u_g)^2-\frac{1}{2}g(\gnabla u_g,\gnabla u_g)g\right)
\endaligned
$$
is divergence-free with respect to $\hat{g}$.

If $M$ is oriented, then we call the associated quadratic differential $P:=T+\I J T$ the \textit{entropy differential}. Since the condition on a symmetric $2$-form on $M$ to be trace-free and divergence free is conformally invariant, we obtain: 
\begin{thm}
Let $(M,g)$ be a smooth oriented Riemannian $2$-mani\-fold with $K_g<0$ and $g$ satisfying the Ricci condition. Then the entropy differential  
$P=T+\I J T$ is holomorphic. 
\end{thm}

\begin{rem}
Note that a metric of negative Gauss-curvature on a surface arising via a constant mean curvature $2$ immersion into hyperbolic $3$-space $\mathbb{H}^3$ also satisfies the Ricci condition (recall that with our convention the `mean' curvature is the sum of the principal curvatures). Besides satisfying the Ricci condition, these so-called \textit{Bryant surfaces} share many properties with minimal surfaces in Euclidean $3$-space, the most important being that they possess a Weierstrass representation~\cite{bryant}. In particular, a quadratic differential similar to the one studied here has been defined for surfaces of constant mean curvature one in hyperbolic three-space $\mathbb{H}^3$ by Bryant~\cite{bryant} and for surfaces of Bryant type in the Lorentz-Minkowski four-space $\mathbb{L}^4$ by Aledo, Galvez and Mira~\cite{MR2200000}. 
\end{rem}

\subsection{The inverse problem}
Suppose we are given a Riemann surface $(M,J)$ and a holomorphic quadratic differential $P$ on $M$.  We ask whether we can locally find a $J$-compatible metric $g$ of negative Gauss curvature on $M$ which satisfies the Ricci condition and so that the entropy differential of $g$ is $P$. 

Let $z: V \to \mathbb{C}$ be local holomorphic coordinates on $(M,J)$. It is easy to check that if the real-valued function $u$ solves Liouville's equation
\begin{equation}\label{liouville}
4\partial^2_{z\bar z}u=e^{-2u},
\end{equation}
then the metric 
$
g=e^{2u}|\d z|^2  
$
satisfies the Ricci condition and $u_g=-\frac{1}{4}\log |K_g|=u$. Now a straightforward computation yields
$$
\aligned
\operatorname{Re}(P)=T&=-{}^{g_0}\mathring{\nabla}^2u-\d u^2 +\frac{1}{2}g_0({}^{g_0}\nabla u,{}^{g_0}\nabla u)g_0\\
&=-2\operatorname{Re}\left(\left(\partial^2_{zz}u+(\partial_z u)^2\right)\d z^2\right)
\endaligned
$$
where $g_0=|\d z|^2$. Writing $P=\frac{\rho}{2}\,\d z^2$ for some holomorphic function $\rho$ on $V$, we are thus interested in the solutions $u$ of the system
$$
4\partial^2_{z\bar z}u=e^{-2u}, \quad \partial^2_{zz} u+(\partial_z u)^2=-\frac{\rho}{4}.
$$
\begin{lem}\label{cplxode}
Let $V\subset \mathbb{C}$ be a simply-connected domain and $\rho$ a holomorphic function on $V$.  We let $z$ be the usual complex coordinate on $\mathbb{C}$.  Then there exist holomorphic functions $w_1, w_2$ on $V$ solving the equation
\begin{equation} \label{linqeqn}
 \partial_{zz}^2 w+\frac{\rho}{4} w=0
\end{equation} 
and with Wronskian satisfying 
\begin{equation} \label{WronskianCond}
W(w_1,w_2)=w_1 \partial_z w_2-w_2 \partial_z w_1=\frac{1}{2}.
\end{equation} 
If $\hat{w}_1, \hat{w}_2$ is another pair of holomorphic solutions to \eqref{linqeqn} satisfying \eqref{WronskianCond}, then there is a unique matrix $B \in \mathrm{SL}(2,\mathbb{C})$
so that $\hat{\wW}=B\wW$ where
\begin{equation*}
 \wW= \begin{pmatrix} w_1 \\ w_2 \end{pmatrix} \mbox{ and }\hat{\wW}=\begin{pmatrix} \hat{w}_1\\ \hat{w}_2 \end{pmatrix}.
\end{equation*}
\end{lem}
\begin{proof}
See for instance~\cite[Chapter 5.2]{Hille}.
\end{proof}

We now have the following:
\begin{prop} \label{SurfaceFromAandTProp}
 Let $V\subset \mathbb{C}$ be a simply-connected domain and suppose that $\rho$ is a holomorphic function on $V$. Then every real-valued function $u\in C^{\infty}(V)$ that satisfies the system
\begin{equation}\label{frobsys}
4 \partial^2_{z\bar{z}}u=e^{-2u}, \quad  \partial_{zz}^2 u+(\partial_z u)^2=-\frac{\rho}{4},
\end{equation}
is of the form
$$
 u_{\wW}=\log |\wW|^2=\log \left(|w_1|^2 +|w_2|^2\right), 
$$
where
\begin{equation*}
 \wW= \begin{pmatrix} w_1 \\ w_2 \end{pmatrix}
\end{equation*}
and $w_1, w_2$ satisfy \eqref{linqeqn} and \eqref{WronskianCond}.  Hence, for each $\rho$ there is a three-dimensional space $\mathrm{SL}(2,\mathbb{C})/\mathrm{SU}(2)$ of solutions $u$.
\end{prop}
\begin{proof}
On $X=V\times \R\times \mathbb{C}$ with coordinates $(z,u,q)$ consider the rank $2$ subbundle $E\subset TX$ defined by the common kernel of the $1$-forms
$$
\varphi_1=\d u -q \d z -\bar q\d \bar z, \quad \varphi_2=\d q +\left(\frac{\rho}{4}+q^2\right)\d z -\frac{1}{4} e^{-2u}\d \bar z.  
$$
Now
$$
\aligned
\d\varphi_1&=\d z\wedge \varphi_2+ \d\bar{z}\wedge \overline{\varphi_2},\\
\d\varphi_2&=-\frac{1}{2} e^{-2u} \d\bar{z}\wedge \varphi_1-2q \d z \wedge \varphi_2 ,\\
\endaligned
$$
hence $E$ is Frobenius integrable. Furthermore, the $1$-graph 
$$
z \mapsto (z,u(z),\partial_z u(z))
$$ of a solution $u$ to \eqref{frobsys} is an integral manifold of $E$. Consequently, a solution $u$ to \eqref{frobsys} is uniquely determined by specifying $u$ and $\partial_z u$ at some point $z_0 \in V$. Simple computations show that for any $(z_0,u_0,q_0) \in X$ there exist holomorphic map $\wW: V \to \mathbb{C}^2$ satisfying (\ref{linqeqn}, \ref{WronskianCond}) so that $u_{\wW}=\log |\wW|^2$ solves \eqref{frobsys} and satisfies 
$$
u(z_0)=u_0, \quad \partial_z u(z_0)=q_0. 
$$
Clearly, if $\hat{\wW}=U\wW$ for $U \in \mathrm{SU}(2)$, then $u_{\hat{\wW}}=u_{\wW}$.
\end{proof}
\begin{cor}\label{MinSurfaceFromHillsEqnCor}
Let $V\subset \mathbb{C}$ be a simply-connected domain and suppose that $\rho$ is a holomorphic function on $V$. Let  
\begin{equation*}
g_{\wW}=|\wW|^4 |\d z|^2 \quad \text{and} \quad A=\operatorname{Re}(\d z^2) 
\end{equation*}
where
\begin{equation*}
 \wW= \begin{pmatrix} w_1 \\ w_2 \end{pmatrix}
\end{equation*}
and $w_1, w_2$ satisfy \eqref{linqeqn} and \eqref{WronskianCond}. Then there is a minimal immersion $\xX_\wW:V\to \Real^3$ with geometric data $(V, g_{\wW}, A)$ and entropy differential $P=\frac{\rho}{2} \d z^2$.
\end{cor}
\begin{proof}
 This is an immediate consequence of Proposition \ref{SurfaceFromAandTProp} and the fundamental theorem of submanifold geometry.
\end{proof}

\section{Weierstrass Representation}\label{WeierstrassSec}
In this section we express the entropy differential $P$ in terms of the Weierstrass data of the minimal surface $\Sigma$ -- this allows us to compute $P$ more readily and to easily analyze its singular and asymptotic behavior.

\subsection{The Weierstrass Representation}
Recall, to an oriented minimal surface $\Sigma$ in $\Real^3$ with parametrization $\xX_\Sigma:M \to \Sigma$ one can associate \emph{Weierstrass data} which encodes the surface and parametrization $\xX_\Sigma$ in complex analytic data. More precisely, the Weierstrass data associated to $\xX_\Sigma$ is the quadruple $(M, J, G, \eta)$ where $(M,J)$ is a Riemann surface, $G$ is a meromorphic function on $(M,J)$ and $\eta$ a holomorphic one form on $(M,J)$.  The data is determined as follows:
\begin{enumerate}
 \item $J$ is the almost-complex structure induced by $\xX_\Sigma$;
\item $G=S\circ \nN$ where $\nN$ is the Gauss map and
$$S:\partial B_1 \backslash (0,0,-1)\to \mathbb{C}$$
 is stereographic projection;
\item $\xX_\Sigma^* \d x_3= \re \eta$.
\end{enumerate}
The Weierstrass data allows one to reconstruct $\xX_\Sigma$ by the means of the \emph{Weierstrass representation}:
\begin{equation} \label{WeierstrassRep}
 \xX_\Sigma(p)-\xX_\Sigma(p_0)= \re \int^p_{p_0} \left( \frac{1}{2} (G^{-1}-G), \frac{i}{2 }( G^{-1}+G) , 1\right) \eta.
\end{equation}
Conversely, given any quadruple $(M, J, G, \eta)$ we may use \eqref{WeierstrassRep} to construct a parametrization $\xX_\Sigma$ of a branched minimal surface $\Sigma$ provided:
\begin{enumerate}
 \item \label{WRCond1}Both $G \eta$ and $G^{-1} \eta$ are holomorphic;
 \item \label{WRCond2}For any 1-cycle $\gamma$  in $M$:
  \begin{equation*}
   \int_\gamma \left( \frac{1}{2} (G^{-1}-G), \frac{i}{2 }( G^{-1}+G) , 1\right) \eta\in i \Real^3.
  \end{equation*}
\end{enumerate}
Condition \eqref{WRCond2} is known as the period condition.
\begin{rem}
 The parametrizing map $\xX_\Sigma$ is an immersion if and only if $G\eta, G^{-1} \eta$, and $\eta$ do not all simultaneously vanish at any point of $M$.
\end{rem}

It is convenient to choose a local complex coordinate patch $(V, z)$ on $M$ and to write $\eta=h\d z$ and $G=G(z)$.  We write $f'$ for $\partial_z f$ for any function $f\in C^{1}(V, \mathbb{C})$.
Standard computations (see for instance~\cite{KarcherNotes}) give
the metric as
\begin{equation*}
 g=\xX_\Sigma^* g_E= \frac{1}{4} (|G|+|G|^{-1})^2 \eta \otimes \overline{\eta}=\frac{|h|^2}{4} (|G|+|G|^{-1})^2 |\d z|^2,
\end{equation*}
the Hopf differential as
\begin{equation*}
Q = -\frac{1}{G}\d G\circ \eta = -\frac{h G'}{G} \d z^2,
\end{equation*}
and the Gauss curvature
\begin{equation*}
  K_{g} =- \frac{16  |G G'|^2}{ |h|^2 (1+|G|^{2})^4}.
\end{equation*}
Hence,
\begin{equation*}
 u_g =-\log 2 -\frac{1}{4} \log |h^{-1} G G'|^2 + \log (1+|G|^2).
\end{equation*}

\subsection{Computing $P$ in terms of Weierstrass data}
We now compute the entropy differential $P$ in terms of the Weierstrass data.
\begin{prop} \label{WeierstrassCompProp}
 Let $\Sigma$ be an oriented minimal surface in $\Real^3$ with Weierstrass data $(M, J, G, \eta)$. If $(U, z)$ is a coordinate chart of $M$ on which  $K_g<0$ and we write $\eta=h\d z$, $G=G(z)$, then $P=\frac{\rho}{2}\d z^2$ with
\begin{align*}
 \rho&= \left(\frac{G'''}{G'}  +\frac{G''}{2G}-\frac{3(G')^2}{4G^2} -\frac{7(G'')^2}{4(G')^2}+\frac{G'' h'}{2G' h} -\frac{G' h'}{2Gh}  -\frac{h''}{h}+\frac{5(h')^2}{4h^2} \right).
\end{align*}
If $Q=\d z^2$, then
\begin{align*}
 P &= \left(\left(\frac{G''}{G'}\right)' -\frac{1}{2} \left( \frac{G''}{G'}\right)^2\right) \d z^2\\
   &= \lbrace G,z\rbrace \d z^2,
\end{align*}
where $\lbrace G, z \rbrace$ is the Schwarzian derivative of $G$.
\end{prop}
\begin{rem}
 The Schwarzian derivative of $G$ has also been studied from a different perspective by Duren, Chuaqui and Osgood~\cite{DurenEtAl} (see also~\cite{MR1656822} for a coordinate free definition of the Schwarzian derivative).
\end{rem}
\begin{proof}
If  $K_g<0$ on $V$, then $\frac{h G'}{G}$ has no zeroes on $V$.  Hence, if $V$ is simply connected there is global square root of $-\frac{hG'}{G}$. Indeed,  there is a function $w$ on $V$ so that
\begin{equation*}
 \d w= \sqrt{-\frac{h G'}{G}} \d z
\end{equation*}
and so
\begin{equation*}
 Q=\d w^2.
\end{equation*}
The exact one-form $\d w$ is well-defined up to multiplication by $\pm 1$.
In particular, we have that the entropy differential is given by
\begin{equation*}
 P= -2 \left( \partial^2_{ww} u_g +(\partial_w u_g)^2\right) \d w^2.
\end{equation*}

In order to express $P$ in terms of the Weierstrass data we note that:
\begin{equation*}
 \partial_w = \sqrt{ -\frac{G}{h G'}} \partial_z
\end{equation*}
and so 
\begin{equation*}
 \partial_{ww}^2 = -\frac{G}{h G'} \partial^2_{zz}-\frac{G}{2 h G' } \left(\frac{G'}{G}-\frac{G'' }{G' } -\frac{h'}{h}\right) \partial_{z}.
\end{equation*}
Hence,
\begin{align*}
 \sqrt{- \frac{h G'}{G}}\partial_w u_g& =-\frac{1}{4}\frac{ (h^{-1} G G')'}{h^{-1} G G'} +\frac{ G' \bar{G}}{1+|G|^2}\\
                              &= \frac{1}{4}\left(\frac{h'}{h} -\frac{G'}{G} -\frac{G''}{G'} \right) +\frac{ G' \bar{G}}{1+|G|^2}
\end{align*}
and
\begin{align*}
 -\frac{h G'}{G}\partial^2_{ww} u_g=& \frac{1}{4} \left( \frac{h''}{h}-\left(\frac{h'}{h}\right)^2-\frac{G''}{G}+\left(\frac{G'}{G}\right)^2-\frac{G'''}{G'}+\left(\frac{G''}{G'}\right)^2\right)+\\&+\frac{G'' \bar{G}}{1+|G|^2}-\left(\frac{G' G}{1+|G|^2}\right)^2+\frac{1}{8} \left( \frac{G''}{G'}\right)^2-\frac{1}{8}\left( \frac{h'}{h}-\frac{G'}{G}\right)^2-\\& -\frac{1}{2}\left( \frac{G''}{G'}+\frac{h'}{h}-\frac{G'}{G} \right) \frac{ G' \bar{G}}{1+|G|^2}.
\end{align*}
We note that both these expressions are independent of replacing $w$ by $-w$ and so hold even if $V$ is not simply-connected.
Combining the above we determine that
$
P=\frac{\rho}{2}\,\d z^2
$
with
\begin{align*}
 \rho&= \left(\frac{G'''}{G'}  +\frac{G''}{2G}-\frac{3(G')^2}{4G^2} -\frac{7(G'')^2}{4(G')^2}+\frac{G'' h'}{2G' h} -\frac{G' h'}{2Gh}  -\frac{h''}{h}+\frac{5(h')^2}{4h^2} \right).
\end{align*}
as claimed.
If $Q=\d z^2$, then
\begin{equation*}
 h=-\frac{G}{G'}
\end{equation*}
and so
\begin{equation*}
 \frac{h'}{h}=\frac{G'}{G}-\frac{ G''}{G'}
\end{equation*}
 and
\begin{equation*}
 \frac{h''}{h}=-\frac{ G''}{G}+2\frac{ (G'')^2}{(G')^2}.
\end{equation*}
Plugging these into the formula for $P$ gives
\begin{align*}
  P &= \left( \left(\frac{G''}{G'}\right)'-\frac{1}{2}\left(\frac{G''}{G'}\right)^2\right) \d z^2.
\end{align*}
\end{proof}

As an application of the previous computation, we determine the behavior  of the entropy differential at umbilic points of $\Sigma$:
\begin{cor}\label{DoublePoleCor}
 If $\Sigma$ is a minimal surface in $\Real^3$ and $p\in \Sigma$ an isolated umbilic point, then $P$, the entropy differential of $\Sigma$, has a double pole at $p$.  Indeed, there is a complex coordinate $z$ around $p$ satisfying $z(p)=0$ and so
\begin{equation*}
 P= -\left( \frac{3n^2+4n}{8} \right) \frac{\d z^2}{z^2} +O(1),
\end{equation*}
where $n$ is the order of vanishing of the Hopf differential $Q$ at $p$.
\end{cor}
\begin{proof}
By rotating $\Sigma$ in $\Real^3$, we may assume that $\nN(p)=\eE_1$ where $(\eE_1,\eE_2,\eE_3)$ denotes the standard basis of $\R^3$. Hence, there is $p$-neighbor\-hood $V$ with a $p$-centered complex coordinate $z$, together with Weierstrass data $(V, J, G, \eta)$ parametrizing $\Sigma$ near $p$ which satisfies $\eta=\d z$ and $ G(z)=1+o(1)$. 
In fact, there are $a,b\in \mathbb{C}$ with $a\neq 0$ so that
\begin{equation*}
 G(z)=1+a z^{n+1} +b z^{n+2} +O(z^{n+2}),
\end{equation*}
because the umbilic point is isolated.
Indeed,
\begin{equation*}
 Q=-\frac{h G'}{G} \d z^2= -a (n+1) z^{n} \d z^2 +O(z^n)
\end{equation*}
and $n$ is the order of vanishing of $Q$ at $p$.

We let $V^*=V\backslash\set{p}$ and apply Proposition \ref{WeierstrassCompProp} to compute that
\begin{equation*}
 P= \frac{1}{2}\left( -\left(\frac{3}{4} n^2+n\right)z^{-2}-\frac{3}{2} \frac{n(n+2)}{n+1} \frac{b}{a} z^{-1} \right) \d z^2 +O(1).
\end{equation*}
However, by changing coordinates to $z\to z+c z^2$ for an appropriate choice of $c$ we obtain $P$ in the desired form.
\end{proof}
We may also use Proposition \eqref{WeierstrassCompProp} to compute the entropy differential at branch points.  
\begin{cor} \label{BranchPtCor}
 Suppose that $(M, J)$ is a Riemann surface and $\xX:M\to \Sigma\subset \Real^3$ is a non-flat branched minimal immersion. Let $p\in M$ be a branch point of $M$ of order $n$ and index $k$.   
\begin{enumerate}
 \item If $n-k+1\neq 0$, then the entropy differential, $P$, has a double pole at $p$ and there is a complex coordinate patch $(V,z)$ about $p$ with $z(p)=0$ so that
\begin{equation*}
 P=\left(\frac{(n+k+1)^2-4 k^2}{8 } \right)\frac{\d z^2}{z^2}+O(1);
\end{equation*}
\item If $n-k+1=0$, then $P$ has at most a simple pole at $p$.
\end{enumerate}
\end{cor}
\begin{proof}
 We may pick a complex coordinate patch $(V,z)$ about $p$ so that $z(p)=0$ and on $V^*=V\backslash \set{p}$ the parameterization $\xX$ is a smooth immersion. Let $(z(V^*), J, G, \eta)$ be the Weierstrass data of this immersion where here $J$ is the usual complex structure.  As $\xX$ has an order $n\geq 1$ branch point with index $k$ at $z(p)=0$, up to an ambient rotation of $\Real^3$ and a re-parameterization the data has the form
\begin{equation*}
 G(z)= z^k
\end{equation*}
for $k\geq 1$ and
\begin{equation*}
 \eta=\left(a z^{n+k}+bz^{n+k+1}\right) \d z+O(z^{n+k+2})
\end{equation*}
where $a\neq 0$.
Computing gives
\begin{align*}
 P&=\left(\frac{(n+k+1)^2-4 k^2}{8 z^2}+\frac{\frac{b}{a}(n+k-1)}{4z}\right) \d z^2+O(1)\\
 &=\left(\frac{(n-k+1)(n+3k+1)}{8 z^2}+\frac{\frac{b}{a}(n+k-1)}{4z}\right) \d z^2+O(1).\\
\end{align*}
The corollary follows by noting that if $n-k+1=0$, then $P$ has at most a simple pole at $p$ as claimed. If $n-k+1\neq 0$, then $P$ has a double pole and may be put in the claimed form by replacing $z$ by $z+cz^2$ for an appropriate choice of $c$.
\end{proof}
\begin{rem}
 We do not distinguish between true and false branch points.  However, any false branch point of a smooth minimal surface at a point with non-vanishing curvature must have order of vanishing $n$ and index $n+1$.
\end{rem}

\subsection{Hill's equation and the (spinor) Weierstrass representation}
We conclude by relating the solutions $w_1,w_2$ from Proposition \ref{SurfaceFromAandTProp} to the Weierstrass data. We observe a connection with the spinorial Weierstrass representation of~\cite{Kusner} but do not explore this in depth.
\begin{prop}\label{WeierstrassHillEqnProp}
Fix a simply-connected domain $V\subset \mathbb{C}$.
Suppose $(V, J_{std}, G, h\,\d z)$ is the Weierstrass data of a minimal immersion with Hopf differential $Q=\d z^2$ and entropy differential $P=\frac{\rho}{2}dz^2$, then 
\begin{align*}
 w_1(z)&=\frac{\sqrt{2}}{2}  \sqrt{-G^{-1}(z) h(z)}\\
 w_2(z)&=\frac{\sqrt{2}}{2}  \sqrt{-G(z) h(z)}
\end{align*}
are single-valued and satisfy \eqref{linqeqn}.  Furthermore, $w_1$ and $w_2$ satisfy \eqref{WronskianCond} provided the branches of the square-root are chosen so $\frac{w_2}{w_1}=G$.  
\end{prop}
\begin{proof}
As $G$ and $h\, \d z$ is the Weierstrass data of a minimal immersion,  $G h$ or $G^{-1} h$ do not have a pole on $V$.  Moreover, if either function vanished at a point $z_0$, then $h(z_0)=0$. As $Q=\d z^2$,  $-\frac{G'}{G} h=1$. Because $G$ is meromorphic,  $h$ has at most a simple zero at $z_0$ and so $G$ has either a simple pole or a simple zero at $z_0$. Hence, at $z_0$ either $G h\neq 0$  and $G^{-1} h$ has a double zero or $G^{-1} h\neq 0$  and $G h$ has a double zero.  Taken together this implies that $w_1$ and $w_2$ are single-valued.

A straightforward computation gives that $w_1, w_2$ satisfy the Wronskian condition~\eqref{WronskianCond}.  Differentiating \eqref{WronskianCond} once, gives that 
$$\frac{w_1''}{w_1}=\frac{w_2''}{w_2}=-\frac{\hat{\rho}}{4}$$ 
for a meromorphic function $\hat{\rho}$.
It is a classical fact -- see for instance~\cite{Hille} -- that if $\hat{w}_1, \hat{w}_2$ solve $w''+\frac{\hat{\rho}}{4} w=0$, then $\hat{G}=\frac{\hat{w}_1}{\hat{w}_2}$ satisfies $\{\hat{G}, z\}=\frac{\hat{\rho}}{2}$. As $\frac{w_2}{w_1}= G$, this implies that $\hat{\rho}=2 \{ G, z \}=\rho$ and so $w_1, w_2$ satisfy \eqref{linqeqn}.
\end{proof}
\begin{cor} \label{WeierstrassFromHillsEqnCor}
Let $V\subset \mathbb{C}$ be a fixed simply-connected domain.
 If $\rho$ is a holomorphic function on $V$ and  
$\wW=(w_1, w_2)^\top$ satisfies \eqref{linqeqn} and \eqref{WronskianCond}, then the minimal immersion $\xX_\wW$ of Corollary \ref{MinSurfaceFromHillsEqnCor} may be be chosen to have Weierstrass data  $\left(V, J_{std},G, \eta\right)$, where
\begin{equation*}
G= \frac{w_2}{w_1}\quad \text{and}\quad \eta=-2w_1 w_2 \d z.
\end{equation*}
\end{cor}
\begin{rem}
 If we let $s_i=w_i \sqrt{\d z} $ be holomorphic spinors, then the $s_i$ are (up to choices of normalization) the spinor Weierstrass data of~\cite{Kusner}.
\end{rem}
\begin{proof}  
Set ${G}=\frac{{w}_2}{{w}_1}$ and  $\eta=-2{w}_1{w}_2 \d z$ and let $\xX_{{\wW}}$ be the  minimal immersion corresponding to this data.
As $w_1=\frac{\sqrt{2}}{2}  \sqrt{-G^{-1}(z) h(z)}$ and $w_2=\frac{\sqrt{2}}{2}  \sqrt{-G(z) h(z)}$, Proposition \ref{WeierstrassHillEqnProp} implies that the entropy differential of $\xX_\wW$ is $\frac{\rho}{2} dz^2$.
A direct computation and \eqref{WronskianCond} imply that the Hopf differential of $\xX_{\wW}$ is $dz^2$.
Finally,
$$
\xX_{\wW}^*g_E=\frac{1}{4} |h|^2\left( |G|+|G^{-1}|\right)^2 |dz|^2=|\wW|^4 \d z \otimes \d\bar{z} =g_{\wW}.
$$
Hence, $\xX_{\wW}$ satisfies the conclusions of Corollary \ref{MinSurfaceFromHillsEqnCor} which verifies the claim.
\end{proof}

\section{Characterization of Minimal Surfaces in Terms of the Entropy Differential}\label{CharSec}
In this section we characterize a number of classical minimal surfaces in terms of the entropy form. In particular, we show that the entropy form vanishes if and only if the surface is contained in Enneper's surface.  The catenoid and helicoid are also characterized in terms of a simple relationship between the entropy form and the second fundamental form.

\subsection{Deformed Catenoids and Helicoids}
In order to get a complete characterization we must introduce two one-parameter families of surfaces, $\mathcal{C}$ and $\mathcal{H}$, which we call, respectively, deformed catenoids and deformed helicoids.
Specifically, $\mathcal{C}$ is the family of surfaces $C_t$ with Weierstrass data 
\begin{equation*}
\left(\mathbb{C}, J, \frac{t-e^z}{1-t e^z}, \frac{1}{1-t^2} (1-t e^{-z})(1-t e^{z})\d z\right). 
\end{equation*}
Similarly, $\mathcal{H}$ is the family of surfaces $H_t$ with Weierstrass data
\begin{equation*}
\left(\mathbb{C}, J, \frac{t- e^{z}}{1- t e^{z}}, \frac{-i}{1-t^2} (1-t e^{-z})(1-t e^{z})\d z\right). 
\end{equation*}
In both cases, $z$ is the usual coordinate on $\mathbb{C}$, $J$ the usual complex structure and $t\in (-1,1)$.  In particular,  $C_0$ is the vertical catenoid and $H_0$ is the vertical helicoid.  Computing as in the preceding section we obtain that for surfaces in $\mathcal{C}$
\begin{equation*}
 P=-\frac{1}{2} \d z^2=\frac{1}{2} Q,
\end{equation*}
and for surfaces in $\mathcal{H}$
\begin{equation*}
 P=-\frac{1}{2} \d z^2=\frac{i}{2} Q.
\end{equation*}
We remark that $\mathcal{C}$ and $\mathcal{H}$ are obtained from $C_0$ and from $H_0$ by applying the one parameter family of M\"{o}bius transforms 
$$
B_t:z\mapsto \frac{t+z}{1-t z}
$$
to the Gauss maps of $C_0$ and $H_0$.

Writing $z=x+iy$ and integrating \eqref{WeierstrassRep} gives the parameterizations of $C_t \in \mathcal{C}$:
\begin{align*}
 \FF_t^{\mathcal{C}}(x,y)&=  \FF_0^{\mathcal{C}}(x,y)+\frac{2 t}{1-t^2} \left( 0,- y+t \cosh x \sin y   , t x-\sinh x \cos y\right) ;\\
\FF_0^{\mathcal{C}}(x,y)&=\left(\cosh x \cos y , \cosh x \sin y, x\right).
\end{align*}
Here $\FF_0^{\mathcal{C}}$ is a parameterization of (an infinite cover of) the catenoid.  By inspection,  $C_t$ has $\Pi_2=\set{x_2=0}$ and $\Pi_3=\set{x_3=0}$ as planes of reflectional symmetry. Moreover, 
$$\FF_t^{\mathcal{C}}(x,y+2\pi)=\FF_t^{\mathcal{C}}(x,y)-\frac{4\pi t}{1-t^2} \eE_2$$
 and so $C_t$ is singly-periodic.
When $t \neq 0$, it is straightforward to see that $C_t$ is not embedded. Suppose $E_\theta$ is the rotation of the upper half of $C_0$ by $\theta$ around the $x_2$-axis. One verifies that $C_t$ is close to the union of translates of $E_\theta$ and of $E_{\pi-\theta}$ where here $\theta=\tan^{-1}\left( \frac{2t}{1-t^2}\right)$.

Similarly, elements of $\mathcal{H}$ are parametrized by
\begin{align*}
 \FF_t^{\mathcal{H}}(x,y)&=  \FF_0^{\mathcal{H}}(x,y)+\frac{2t}{1-t^2} \left(0,  x+t\sinh x \cos y, t y- \cosh x \sin y\right);\\
\FF_0^{\mathcal{H}}(x,y)&= (\sinh x \sin y,- \sinh x \cos y, y).
\end{align*}
For $t=0$ this is a parametrization of the helicoid. Note that the image of $\set{x=0}$ is a the $x_3$-axis while the the image of $\set{y=n \pi}$ for $n$ an integer are the set of parallel lines $\set{x_1=0, x_3=\frac{1+t^2}{1-t^2} n \pi}$ contained in the $\set{x_1=0}$ plane.  Moreover, 
$$\FF_t^{\mathcal{H}}(x,y+2\pi)= \FF_t^{\mathcal{H}}(x,y)+2\pi \frac{1-t^2}{1+t^2} \eE_3$$
 so $H_t$ is singly-periodic.
 For $t\neq 0$, $H_t$ is not embedded. However, if we denote by $H^\pm_t$ the two components of $H_t \backslash \set{x_1=x_2=0}$, then each $H^\pm_t$ is embedded.  In fact, each is a multi-valued graphs over the plane $\Pi_\theta$ which contains the $x_2$-axis and makes an angle $\theta=\tan^{-1}\left( \frac{2t}{1-t^2}\right)$ with the plane $\Pi_3=\set{x_3=0}$.  In particular, rotating $H^\pm_t$ by $\theta$ around the $x_2$-axis gives a surface that looks (roughly) like a sheared copy of $H_0^\pm$.

\subsection{Characterization of minimal surfaces in terms of $P$ and $Q$}
We now characterize surfaces in terms of simple relationships between $P$ and $Q$. In light of Proposition \ref{SurfaceFromAandTProp}, we expect there to be a three-parameter family of surfaces for any fixed of $P$ and $Q$.  However, in simple settings two of these parameters correspond to re-parameterizations. 
\begin{thm}\label{MainCharacterizationThm}
 Let $\Sigma$ be a smooth oriented non-flat minimal surface in $\Real^3$ with Hopf differential $Q$ and entropy differential $P$. We have:
\begin{enumerate}
 \item \label{EnneperCase} If $P\equiv 0$, then up to a rigid motion and homothety, $\Sigma$ is contained in Enneper's surface;
 \item \label{CatCase} If $\lambda \neq 0$ and  $P\equiv \lambda Q$, then,  up to a rigid motion and homothety, $\Sigma$ is contained in a surface $C\in\mathcal{C}$.  If $\Sigma$ is properly embedded, then it is the catenoid;
\item \label{HelCase} If $\lambda \neq 0$ and $P \equiv  \I Q$, then, up to a rigid motion and homothety, $\Sigma$ is contained  in a surface $H\in \mathcal{H}$.  If $\Sigma$ is properly embedded, then it is the helicoid.
\end{enumerate}
\end{thm}
\begin{rem}
If $\Sigma$ is an oriented minimal surface in $\R^3$ with Hopf differential $Q$ and entropy differential $P$, then for any $\lambda>0$ the rescaling scaling of $\lambda \Sigma$ has Hopf differential $\lambda Q$ and entropy differential $P$.  Reversing the orientation of $\Sigma$ changes $Q$ to $-Q$ but leaves $P$ unchanged.  
\end{rem}
\begin{proof}
After possibly rescaling $\Sigma$ and reversing the orientation, we may assume that $P=-\frac{1}{2} \alpha^2 \d z^2$ where $\alpha=0$ in Case \eqref{EnneperCase}, $\alpha^2=1$ in Case \eqref{CatCase} and $\alpha^2=i$ in Case \eqref{HelCase}.
As $\Sigma$ is smooth and non-flat, the second fundamental form has no singularities and $P$ can only have isolated singularities. Hence, by Corollary \ref{DoublePoleCor}, in all cases $P$ has no singularities and $Q$ has no zeros on $\Sigma$.  
Hence, for any point $p\in \Sigma$ there is a simply connected neighborhood $V$ of $p$ and complex coordinate $z:V \to \mathbb{C}$ so that the Hopf differential satisfies $Q=-\d z^2$. That is, $P=\frac{\alpha^2}{2} dz^2$.
By Corollary \ref{WeierstrassFromHillsEqnCor}, in order to recover the surface it is enough to understand the holomorphic solutions on  $z(V)$ to the Hill's equation:
\begin{equation}\label{HillsEqnConst}
 \partial^2_{zz}w-\frac{\alpha^2}{4}w =0.
\end{equation}
Clearly, this equation makes sense on all of $\mathbb{C}$ (with $z$ as the usual coordinate) and analytic continuation implies that all solutions are obtained by restricting global solutions to $z(V)$.  Let $\wW(z)=(w_1(z), w_2(z))^\top$ be a pair of solution to the Hill's equation with Wronskian $W(w_1,w_2)=\frac{1}{2}$. 

We note there are two natural actions on the space of solutions.  The first is the natural action of $\mathrm{SL}(2,\mathbb{C})$ of Proposition \ref{SurfaceFromAandTProp} which is transitive.  The second is an action of $\mathbb{C}$ that arises from the translation invariance of \eqref{HillsEqnConst}.  Specifically, let $\mathbb{C}$ act on $\wW$ by $\tau\mapsto \wW(z+\tau)$. The translation invariance of \eqref{HillsEqnConst} and of the Wronskian condition implies that this is a well defined action.  By Proposition \ref{SurfaceFromAandTProp}, the action of $\mathrm{SU}(2)\subset \mathrm{SL}(2,\mathbb{C})$ does not change the geometry of the surface. Likewise, the action of $\mathbb{C}$ amounts to a change of coordinates and also does not change the geometry.  Our goal is to determine all geometrically distinct solutions.

First, note that the Gram-Schmidt procedure implies that any matrix $B\in \mathrm{SL}(2,\mathbb{C})$ may be factored as
\begin{equation*}
 B=U L
\end{equation*}
where $U\in \mathrm{SU}(2)$ and $L\in \mathrm{SL}(2,\mathbb{C})$ is lower triangular with positive entries on the diagonal and $\det L=1$.  This is sometimes called the $QR$ (or in this case $QL$) factorization. 
We write any such $L$ as
\begin{equation*}
L=\begin{bmatrix} \mu & 0 \\ \nu & \mu^{-1} \end{bmatrix}
\end{equation*}
where $\mu>0$ and $\nu \in \mathbb{C}$.
We now treat the case $\alpha=0$ and $\alpha \neq 0$ separately.
{\bf Case \eqref{EnneperCase}:}
By inspection a pair of solutions to \eqref{HillsEqnConst} with $\alpha=0$ and satisfying the Wronskian condition are
\begin{equation*}
 w_1(z)=1 \mbox{ and } w_2(z)=\frac{1}{2} z.
\end{equation*}
Hence, by the $QL$ factorization, the functions
\begin{equation*}
 w_1(z)=\mu \mbox{ and } w_2(z)=\nu+\frac{1}{2}\mu^{-1} z,
\end{equation*}
with $\mu>0$ and $\nu\in \mathbb{C}$ give all geometrically distinct solutions to \eqref{HillsEqnConst}.
Applying the translation action with  $\tau=-2\mu \nu $
gives all geometrically distinct solutions in the form
\begin{equation*}
 w_1(z)=\mu \mbox{ and } w_2(z)= \frac{1}{2}\mu^{-1} z.
\end{equation*}
By Corollary \ref{WeierstrassFromHillsEqnCor} the Gauss map the associated minimal surfaces maybe chosen so
\begin{equation*}
G(z)=\frac{w_2}{w_1}=\frac{ z}{2\mu^2}.
\end{equation*}
Moreover, as $Q=-\d z^2$ the height differential is $\eta=z \d z$. This is precisely the Weierstrass data of a rescaling of Enneper's surface proving the claim in this case.

{\bf Case \eqref{CatCase} and \eqref{HelCase}:}
As $\alpha\neq 0$, a pair of solutions to \eqref{HillsEqnConst} that satisfy the Wronskian condition are
\begin{equation*}
 w_1(z)=\frac{1}{\sqrt{2}\alpha } e^{-\frac{\alpha}{2} z} \mbox{ and } w_2(z)=\frac{1}{\sqrt{2}\alpha} e^{\frac{\alpha}{2} z}.
\end{equation*}
Hence, by the $QL$ factorization, we may write all geometrically distinct solutions  to \eqref{HillsEqnConst} in the form
\begin{equation*}
 w_1(z)= \frac{\mu}{\sqrt{2}\alpha } e^{-\frac{\alpha}{2} z} \mbox{ and } w_2(z)= \frac{1}{\sqrt{2}\alpha }\left(\nu e^{-\frac{\alpha}{2} z}  +\mu^{-1} e^{\frac{\alpha}{2} z}\right) 
\end{equation*}
with $\mu>0$ and $\nu\in \mathbb{C}$.
The translation action
allows us to express all geometrically distinct solutions as
\begin{equation*}
 w_1(z)=i\frac{e^{-i \theta/2}}{\sqrt{2}\alpha}  e^{-\frac{\alpha}{2} z} \mbox{ and } w_2(z)=i\frac{e^{i \theta/2}}{\sqrt{2}\alpha} \left( \gamma e^{-\frac{\alpha}{2} z}  - e^{\frac{\alpha}{2} z}\right) .
\end{equation*}
where $\gamma\geq 0$ and $\theta\in [0,2\pi)$.  Indeed, either $\nu=0$ and we take $\gamma=\theta=0$ or $\nu\neq 0$ and we write $\nu=\gamma \mu e^{i\theta}$. In both cases, we act by $\tau=\frac{1}{\alpha} i(\theta-{\pi}) +\frac{2}{\alpha} \ln \mu$.
Let $\phi\in (-\pi/4,\pi/4)$ satisfy
\begin{equation*}
 \tan 2 \phi = \gamma.
\end{equation*}
The matrix
\begin{equation*}
 \begin{pmatrix} \cos \phi & -\sin \phi \\ \sin \phi & \cos \phi \end{pmatrix} \begin{pmatrix} -i e^{i\theta/2} & 0 \\ 0 & -ie^{-i\theta/2} \end{pmatrix} 
\end{equation*}
is the product of two elements of $\mathrm{SU}(2)$ and so is in $ \mathrm{SU}(2)$.  Acting by this matrix gives that all geometrically distinct solutions can be put in the form
\begin{equation*}
 w_1(z)=\frac{\cos \phi }{\sqrt{2}\alpha \cos 2 \phi}  e^{-\frac{\alpha}{2} z}-\frac{\sin \phi }{\sqrt{2}\alpha }  e^{\frac{\alpha}{2} z}
\end{equation*}and
\begin{equation*} w_2(z)=\frac{\sin \phi}{\sqrt{2}\alpha \cos 2 \phi} e^{-\frac{\alpha}{2} z}  - \frac{\cos \phi }{\sqrt{2}\alpha }e^{\frac{\alpha}{2} z},
\end{equation*}
where $\phi\in (-\pi/4,\pi/4)$. By applying the translation action with $\tau= -\frac{1}{\alpha} \ln \cos 2\phi$, all  geometrically distinct solutions can be put in the simplified form
\begin{equation*}
 w_1(z)=\frac{\cos \phi e^{-\frac{\alpha}{2} z}-\sin \phi   e^{\frac{\alpha}{2} z} }{\sqrt{2 \cos 2 \phi}\alpha}  \mbox{ and } w_2(z)=\frac{\sin \phi e^{-\frac{\alpha}{2} z}  - \cos \phi e^{\frac{\alpha}{2} z} }{\sqrt{2 \cos 2 \phi}\alpha}.
\end{equation*}
By Corollary \ref{WeierstrassFromHillsEqnCor} the Gauss map the associated minimal surfaces may be chosen so
\begin{equation*}
G(z)=\frac{w_2}{w_1}= \frac{ \tan\phi - e^{\alpha z}}{1- \tan \phi e^{\alpha z}}.
\end{equation*}
Set $t= \tan \phi$.
If $\alpha^2=1$, then we may take $\alpha=1$ and as $Q=-\d z^2$ we see that $\eta = \frac{1}{1-t^2} (1-t e^{-z})(1-t e^{z})\d z$ which together with $G(z)$ is precisely the data of a deformed catenoid.
If $\alpha^2=i$, then we write $\zeta=\alpha z$.  In this case $Q=-\frac{1}{\alpha^2} \d\zeta^2=i \d\zeta^2$ and so
$\eta=- \frac{i}{1-t^2} (1-t e^{-\zeta})(1-t e^{\zeta})\d\zeta$ which together with $G(\zeta)$ is precisely the data of a deformed helicoid.
\end{proof}

\section{Curvature Estimates for Embedded Minimal Surfaces in Terms of $T$}
An interesting problem is to make the characterizations of Theorem \ref{MainCharacterizationThm} effective.  For instance, to show that a 
minimal surface with ``small'' entropy form must be close to a rescaling of a piece of Enneper's surface.  A major challenge is to determine 
an appropriate notion of smallness for the entropy form -- something made more difficult by the need to account for the possible 
singularities of $T$.  We propose a certain family of quantities as natural ways to measure this smallness and as an application give a 
novel curvature estimate for embedded minimal surfaces.

Before introducing them we note the following consequence of Corollary \ref{DoublePoleCor}. 
\begin{lem} \label{weightedEntLem}
 Let $\Sigma$ be a smooth minimal surface with metric $g$ and entropy form $T$. For $\alpha>0$ we define,  $\hat{T}_\alpha$,  the 
\emph{$\alpha$-weighted entropy form} of $\Sigma$ by  $\hat{T}_\alpha \equiv 0$ if $\Sigma$ is flat and by
\begin{equation*}
 \hat{T}_\alpha=|K_g|^{\alpha} T
\end{equation*}
otherwise.
In either case, the function $|\hat{T}_\alpha|_g^{\frac{1}{1+\alpha}}$ is locally integrable on $\Sigma$.
\end{lem}
\begin{proof}
 If $\Sigma$ is flat then there is nothing to prove as $\hat{T}^\alpha$ is identically zero.  Otherwise, by Corollary \ref{DoublePoleCor}, 
$T$ is smooth away from the isolated poles where $K_g$ has a zero, in particular $|\hat{T}^\alpha|_g^{\frac{1}{1+\alpha}}$ is locally 
integrable away from the zero set. As $\Sigma$ is smooth and $K_g\leq 0$, if $K_g(p)=0$ at a point $p$, then $\nabla_g K_g (p)=0$.  In 
particular $K_g=O(r^2)$ where $r$ is the distance to $p$. On the other hand, by Corollary \ref{DoublePoleCor}, $P$ has a double pole at $p$ 
and so $|T|_g=C r^{-2} +o(r^{-2})$ for some constant $C\neq 0$.  Hence, 
$|\hat{T}_\alpha|_g^{\frac{1}{1+\alpha}}=O\left(r^{\frac{2\alpha-2}{1+\alpha}}\right)$.  As $\frac{2\alpha-2}{1+\alpha}>-2$ for 
$\alpha>0$, $|\hat{T}_\alpha|_g^{\frac{1}{1+\alpha}}$ is integrable in a neighborhood of $p$.  Since $p$ was an arbitrary singularity of 
$T$, this proves the lemma. 
\end{proof}

We propose that a reasonable notion of size for the entropy differential $T$ of a smooth minimal surface $\Sigma$ is given by
\begin{equation*}
 ||T||_{\Sigma, \alpha}:= 2^{\frac{1}{2(\alpha+1)}}\int_{\Sigma} |\hat{T}_\alpha|_g^{\frac{1}{1+\alpha}} \vol= 2^{\frac{1}{2(\alpha+1)}} \int_{\Sigma} 
|T|_g^{\frac{1}{1+\alpha}} 
|K_g|^{\frac{\alpha}{1+\alpha}} \vol.
\end{equation*}
If $\Sigma_0\subset \Sigma$, then we obviously have a domain monotonicity property
\begin{equation*}
  ||T||_{ {\Sigma_0, \alpha}}\leq  ||T||_{{\Sigma, \alpha}}.
\end{equation*}
By Lemma \ref{weightedEntLem}, if $\Sigma$ is a smooth minimal surface and $\Sigma_0$ is pre-compact in $\Sigma$, then
\begin{equation*}
   ||T||_{\Sigma_0, \alpha}<\infty
\end{equation*}
Finally, if $\hat{T}^\Sigma_\alpha$ is the $\alpha$-weighted entropy form of $\Sigma$, then  $\hat{T}^{\lambda \Sigma}_\alpha=\lambda^{-2\alpha} 
\hat{T}^\Sigma_\alpha$ is the $\alpha$-weighted entropy form of $\lambda \Sigma$.  To see this observe that $T$ is scale invariant (by construction) and the Gauss curvature scales like $\lambda^{-2}$.  Hence, as the norm of a (fixed) quadratic differential scales like $\lambda^{-2}$, 
\begin{equation*}
|\hat{T}_\alpha^{\lambda\Sigma}|_{\lambda g}^{\frac{1}{1+\alpha}} = \lambda^{-2} |\hat{T}^\Sigma_{\alpha}|_g^{\frac{1}{1+\alpha}} \mbox{ and 
so } ||{T}^{\lambda \Sigma}||_{\lambda \Sigma,\alpha} =||{T}^\Sigma||_{\Sigma, \alpha}
\end{equation*}
for all $\lambda>0$. That is, these quantities are scale invariant for all $\alpha>0$.
\begin{rem}
 Clearly, if $\Sigma$ has an umbilic point, then $\lim_{\alpha\to 0} ||T||_{\Sigma, \alpha}=\infty$.  Nevertheless, the normalized 
value $\tau:=\lim_{\alpha\to 0 } \alpha ||T||_{\Sigma, \alpha}$ is finite on reasonable surfaces. 
\end{rem}

\subsection{The Curvature Estimate}
We now use the scale invariance of $||T||_{\Sigma,\alpha}$ and Theorem \ref{MainCharacterizationThm} to prove an 
$\epsilon$-regularity result:
\begin{thm}\label{CurvEstThm}
 There are constants $\epsilon=\epsilon(\alpha)>0$ and $C=C(\alpha)>0$ so that: if $\Sigma$ is a properly embedded minimal surface in 
$B_{2R}$ and 
$$||T||_{\Sigma, \alpha}<\epsilon,$$
 then
\begin{equation*}
 R^2 \sup_{B_{R}\cap \Sigma} |A|^2 \leq C^2.
\end{equation*}
\end{thm}
\begin{rem}
 The embeddedness condition is essential as can by seen by considering an appropriate rescaling of Enneper's surface. However, as 
$\alpha\to \infty$, $||T||_{\Sigma, \alpha}\to \int_{\Sigma} |K_g| \vol$ the total curvature of $\Sigma$.  In this case, the above theorem 
 holds \emph{without} the assumption of embeddedness -- see White~\cite{White} or Anderson~\cite{Anderson}.
\end{rem}

We begin with a Lemma which is crucial to the blow-up argument.
\begin{lem}\label{BlowUpLem}
  Fix $C>0$, $p \in \R^3$ and suppose $\Sigma$ is a properly embedded smooth surface in $B_{2R}(p)\subset\Real^3$ satisfying 
\begin{equation*}
\sup_{{B}_{R}(p)\cap \Sigma} |A|^2 \geq 16 C^2 R^{-2}. 
\end{equation*}
  Then there is a point $q\in \Sigma$ and scale $s>0$ so that $B_{Cs}(q)\subset {B}_{2R}(p)$ and
\begin{equation*}
 \sup_{B_{Cs}(q)\cap \Sigma} |A|^2\leq 4 s^{-2}=4|A|^2(q).
\end{equation*}
\end{lem}
\begin{proof}
 With $r(x)=|x-p|$ define the function
\begin{equation*}
 F(x)=\left(r(x)-\frac{3}{2}R\right)^2 |A|^2.
\end{equation*}
This is a Lipschitz function on $B_{\frac{3}{2}R}(p)\cap \Sigma$ that vanishes on $\partial B_{\frac{3}{2}R}(p)\cap \Sigma$.  As $F$ is continuous, non-negative and vanishes on $\partial B_{\frac{3}{2}R}(p)\cap \Sigma$, $F$ achieves its positive maximum at a point $q\in B_{\frac{3}{2}R}(p)\cap \Sigma$. 

The lower bound 
$$\sup_{{B}_{R}(p)\cap \Sigma} |A|^2 \geq 16 C^2 R^{-2} $$ implies that
$ F(q)\geq 4C^2.$
Set $s = |A|^{-1}(q)$ and
$ \sigma=\frac{3}{2}R- r(q) $
and note that $2 C s\leq \sigma$. Furthermore, if $x\in B_{\sigma/2}(q)$, then 
$$r(x)\leq \frac{3}{2} R-\frac{\sigma}{2}<\frac{R}{2}$$ and so $\sigma^2 \leq 4(r(x)-\frac{3}{2}R)^2$ and $B_{\sigma/2}(q)\subset B_{\frac{3}{2} R}(p)$. Combining these facts,
\begin{equation*}
 \sup_{B_{Cs}(q)\cap \Sigma} \frac{\sigma^2}{4} |A|^2\leq  \sup_{B_{\sigma/2}(q)\cap \Sigma} \frac{\sigma^2}{4} |A|^2\leq\sup_{B_{\sigma/2}(q)\cap \Sigma} F\leq F(q)=\sigma^2 |A|^2(q).
\end{equation*}
Which verifies the claim.
\end{proof}

We also note the following well-known fact:  
\begin{prop} \label{GeomArzAscoProp}
 Suppose that $R_i \nearrow \infty$ and that $\Sigma_i$ are properly embedded minimal surfaces in $B_{R_i}$ so that
\begin{enumerate}
 \item $0\in \Sigma_i$ and $|A^{\Sigma_i}|(0)=1$;
 \item $ \sup_{\Sigma_i} |A^{\Sigma_i}|\leq C<\infty;$
\end{enumerate}
then up to passing to a subsequence, the $\Sigma_i$ converge smoothly and with multiplicity one to a properly embedded minimal surface $\Sigma$ in $\Real^3$ so that $0\in \Sigma$ satisfies $|A^\Sigma|(0)=1$.
\end{prop}
\begin{proof}
 Up to passing to a subsequence, the $\Sigma_i$ converge to a smooth  minimal lamination $\mathcal{L}$ of $\Real^3$. As $0\in \Sigma_i$ for each $i$,  there is a leaf $L$ of the lamination containing $0$, moreover $|A^L|(0)=1$ and so $L$ is not flat.  Furthermore, $\sup_{\Sigma} |A^L| \leq C <\infty$ and so the injectivity radius of $L$ is positive. Hence, by \cite{MRDuke}, $L$ is properly embedded.  Finally, if the convergence is with multiplicity greater than one, then $L$ would be stable and hence flat by~\cite{Fischer-Colbrie1980}.
\end{proof}

\begin{proof}[Proof of Theorem \ref{CurvEstThm}]
By rescaling we may take $R=1$.  Assume the theorem is false, then there is a sequence of minimal surfaces $\Sigma_i$ properly embedded 
in $B_2$ so  $||{T}^{\Sigma_i}||_{\Sigma_i, \alpha}\to 0$ and $\sup_{B_{1}\cap \Sigma_i} |A^\Sigma_i|^2 \to \infty$.
By Lemma \ref{BlowUpLem}, there exist a sequence of $C_i\to \infty$, points $q_i \in \Sigma_i$ and scales $s_i\to 0$ so
$B_{C_i s_i}(q_i)\subset B_{2}$ and
 \begin{equation*}
 \sup_{B_{C_is_i}(q_i)\cap \Sigma_i} |A^{\Sigma_i}|^2\leq 4 s^{-2}_i=4|A^{\Sigma_i}|^2(q_i).
\end{equation*}
We set $\hat{\Sigma}_i=s^{-1}_i(\Sigma_i\cap B_{C s_i}(q_i)- q_i )$. The scaling properties of  $||{T}||_{\Sigma, \alpha}$ and  domain 
monotonicity together imply that 
$$||{T}^{\hat{\Sigma}_i}||_{\Sigma_i, \alpha}\leq ||{T}^{\Sigma_i}||_{\Sigma_i, \alpha}\to 0.
$$
Moreover, each $ \hat{\Sigma}_i$ is properly embedded in $B_{C_i}$, contains $0$ and satisfies
 \begin{equation*}
 \sup_{B_{C_i}(0)\cap \hat{\Sigma}_i} |A^{\hat{\Sigma}_i}|^2\leq 4 =4|A^{\hat{\Sigma}_i}|^2(0).
\end{equation*}
Hence, by Proposition \ref{GeomArzAscoProp}, up to passing to a subsequence, the $\hat{\Sigma}_i$ converge to a smooth properly embedded minimal surface $\hat{\Sigma}$ in $\Real^3$. The convergence is with multiplicity one and $0\in \hat{\Sigma}$ satisfies 
$|A^{\hat{\Sigma}}|^2(0)=1$.  By the smoothness of $\hat{\Sigma}$ and the monotonicity formula, there is a $\rho>0$ so that in $B_{\rho}(0)\cap \hat{\Sigma}$ one has $|A^{\hat{\Sigma}}|> \frac{1}{2}$ and so $\pi \rho^2<Area(\hat{\Sigma}\cap B_{\rho})\leq 2\pi \rho^2$. As the 
$\hat{\Sigma}_i$ converge smoothly and with multiplicity one to $\hat{\Sigma}$, there is an $i_0$ large so that  $i>i_0$ implies   
$|A^{\hat{\Sigma_i}}|> \frac{1}{4}$ and $\frac{\pi}{2} \rho^2 <Area(\hat{\Sigma}\cap B_{\rho})\leq 3\pi \rho^2$. As $A^{\Sigma_i}\neq 0$  , 
$\hat{T}^{\hat{\Sigma}_i}_\alpha$ is smooth in $B_{\rho}\cap \hat{\Sigma}_i $ for $i>i_0$ and converges smoothly to 
$\hat{T}^{\hat{\Sigma}}_\alpha$ in $B_{\rho}$.
However,
 $||{T}^{\hat{\Sigma}_i}||_{B_\rho\cap \hat{\Sigma}_i, \alpha} \to 0$, hence $\hat{T}^{\hat{\Sigma}}_\alpha\equiv 0$ on $B_{\rho}\cap 
\hat{\Sigma}$.  Together with $A^{\hat{\Sigma}}\neq 0$ on $B_\rho\cap \hat{\Sigma}$ this implies $T^{\hat{\Sigma}}_\alpha\equiv 0$ on 
$B_{\rho}\cap \hat{\Sigma}$ and so $B_{\rho}\cap \hat{\Sigma}$ is contained in a rescaled Enneper's surface by Theorem 
\ref{MainCharacterizationThm}. It then follows from the strong unique continuation property of smooth minimal surfaces  that $\hat{\Sigma}$ 
is a rescaled Enneper's surface in $\R^3$, contradicting that $\hat{\Sigma}$ is properly embedded and proving the theorem.
\end{proof}

\subsection{Gap properties of the entropy form}\label{GapSec}
In light of Theorem \ref{CurvEstThm} an interesting question is to determine the optimal constant $\epsilon$ in Theorem \ref{CurvEstThm}.  
This is equivalent to determining a lower bound for $||T||_{\Sigma,\alpha}$ when $\Sigma$ is a non-flat properly embedded minimal surface 
in $\Real^3$.  We present some partial results in this direction as well as pose a question about the expected behavior.

A consequence of Theorem \ref{CurvEstThm} and~\cite{MeeksIIIa} is that if $||T||_{\Sigma,\alpha}$ is finite on a properly embedded surface, 
then the surface has finite total 
curvature.  
\begin{prop}\label{FTCProp}
 If $\Sigma$ is a properly embedded minimal surface in $\Real^3$ and 
$$ ||{T}||_{\Sigma,\alpha}<\infty,$$
then 
$$\int_{\Sigma} |A|^2 \vol=2\int_\Sigma |K_g| \vol<\infty.$$ 
\end{prop}
\begin{proof}
 Since $||{T}||_{\Sigma, \alpha}<\infty$ there is a value $R>0$ so that
$
||T||_{\Sigma\backslash \bar{B}_R,\alpha} <\epsilon
$ where $\epsilon$ is given by Theorem \ref{CurvEstThm}. This implies that there is a constant $C>0$ given by Theorem \ref{CurvEstThm} so 
that 
for $p\in \Sigma \backslash \overline{B}_{2R}$ we have $B_{\frac{1}{2}|p|}{p}\subset \Real^3\backslash B_{R}$ and so
\begin{equation*}
 |A|^2(p)\leq \frac{16 C^2}{|p|^2}.
\end{equation*}
That is, $\Sigma$ has quadratic extrinsic decay of curvature.  Hence, by Theorem 1.3 of~\cite{MeeksIIIa}, $\Sigma$ has finite total 
curvature. 
\end{proof}

In order to get a more refined result, we first compute:
\begin{lem}\label{CatLem}
 If $C$ is the catenoid, then 
\begin{equation*}
 ||{T}||_{C, \alpha}=2\pi^{3/2} 
\frac{\Gamma\left(\frac{\alpha}{1+\alpha}\right)}{\Gamma\left(\frac{1}{2}+\frac{\alpha}{1+\alpha}\right)}. 
\end{equation*}
Here $\Gamma(x)$ is the Gamma function.  Hence,
\begin{equation*}
 \lim_{\alpha\to 0} \alpha ||{T}||_{C, \alpha}=2 \pi \quad \mbox{ and }\quad 
\lim_{\alpha \to \infty} ||{T}||_{C, \alpha}=4\pi.
\end{equation*}
\end{lem}
\begin{proof}
The Weierstrass data for the catenoid is $(\mathbb{C}/\langle 2\pi i\rangle, J, - e^z, \d z)$ where $J$ is the usual complex structure. Using this data and writing $z=x+iy$, we have 
$$
P=-\frac{1}{2} \d z^2, \quad g=\cosh^2 x \,|\d z|^2, \quad K_g=-
\frac{1}{\cosh^4 x}.
$$  
As a consequence, $|T|_g=\frac{\sqrt{2}}{2}|P|_g=\frac{\sqrt{2}}{2\cosh^2 x}$ and so 
$|\hat{T}_\alpha|_g=\frac{\sqrt{2}}{2\cosh^{2+4\alpha} x}$, 
hence
\begin{equation*}
 ||{T}||_{C, \alpha}= 2\pi \int_{-\infty}^\infty \frac{1}{\cosh^{\frac{2\alpha}{\alpha+1}} x} 
dx
\end{equation*}
and the integral was evaluated using Mathematica.
\end{proof}
More generally, we have:
\begin{prop}\label{EndUmbilicProp}
Let $\Sigma$ be a non-flat properly immersed minimal surface in $\R^3$.
 If $E\subset \Sigma$ is an embedded end of finite total curvature with branching order $n\geq 0$, then
$$
  \lim_{\alpha\to 0} \alpha ||{T}||_{E, \alpha}\geq \frac{  (3n+2)(n+2)}{4(n+1)}\pi.
$$
If $U\subset \Sigma$ is open and $U$ contains an umbilic point of order $n\geq 
1$, then
$$
 \lim_{\alpha\to 0} \alpha  ||{T}||_{U, \alpha}\geq \frac{(3n+4)n }{4(n+1)}\pi .
$$
If $\bar{E}\subset \Sigma$ and $\bar{E}$ contains no umbilic points, then we may replace the inequality by an equality. 
Likewise, if 
$\bar{U}\subset \Sigma$ and $\bar{U}$ contains only the one umbilic point, then we may replace the 
inequality by an equality.
\end{prop}
\begin{proof}
We begin with a general computation.
Let $\mathbb{D}^*=\mathbb{D}\backslash \set{0}$ be the punctured disk with the usual complex coordinate $z=re^{i \theta}$. Suppose that 
$\Sigma$ is a minimal surface conformally parametrized by $\mathbb{D}^*$ and the following asymptotics hold for the entropy 
differential, 
metric and Gauss curvature as $r\to 0$
$$
P= \frac{\beta}{2 }\frac{\d z^2}{z^2} +O\left(\frac{1}{r^3}\right) \d z^2, \qquad g=\mu r^k \d z \otimes \d \bar{z} +O\left(r^{k+1}\right) 
\d z \otimes \d \bar{z},
$$
and
$$
K_g=-\gamma r^l +O\left(r^{l+1}\right),
$$
where $\beta \in \mathbb{R}^*$ and $\mu,\gamma, k+l+2>0$.
As 
$$
|T_g| = \frac{1}{\sqrt{2}} |P|_g=\frac{|\beta|}{ \sqrt{2} \mu} r^{-(k+2)} +O\left( r^{-(k+1)}\right),
$$ we compute that
\begin{align*}
 |\hat{T}_g|^{\frac{1}{1+\alpha}}  \sqrt{|g|} &=
  \left( \frac{|\beta|}{ \sqrt{2} \mu} r^{-(k+2)} +O\left( 
r^{-(k+1)}\right)\right)^{\frac{1}{1+\alpha}}\cdot \left( \gamma r^l +O\left(r^{l+1}\right)\right)^{\frac{\alpha}{1+\alpha}} \\
 &\phantom{=}\; \cdot \left( 
\mu r^k+O\left(r^{k+1}\right)\right) \\
&= 2^{-\frac{1}{2(1+\alpha)}} \mu^{1-\frac{1}{1+\alpha}}  |\beta|^{\frac{1}{1+\alpha}} \gamma^{\frac{\alpha}{1+\alpha}}
r^{\frac{-1+(k+l+1)\alpha}{1+\alpha}-1} +O\left(r^{\frac{-1+(k+l+1)\alpha}{1+\alpha}}\right).
\end{align*}
Picking $R_0>0$ so the asymptotic bounds hold for $|z|\leq R_0$,  gives
\begin{align*}
 ||{T}||_{\Sigma, \alpha} &\geq 2^{\frac{1}{2(1+\alpha)}} \int_0^{2\pi} \int_0^{R_0}|\hat{T}_g|^{\frac{1}{1+\alpha}}  \sqrt{|g|}  r 
\d r \d \theta \\
&= 2\pi \mu^{1-\frac{1}{1+\alpha}}  |\beta|^{\frac{1}{1+\alpha}} \gamma^{\frac{\alpha}{1+\alpha}}\int_0^{R_0} 
r^{\frac{-1+(k+l+1)\alpha}{1+\alpha}} dr +\int_0^{R_0} O\left(r^{\frac{-1+(k+l+1)\alpha}{1+\alpha}+1}\right) dr.
\end{align*}
Evaluating the integrals, we conclude that
\begin{equation}\label{Compu}
\lim_{\alpha \to 0} \alpha ||{T}||_{\Sigma, \alpha}\geq \frac{ 2\pi |\beta| }{k+l+2}.
\end{equation}
We can replace inequality by equality provided the bounds hold on all of $\mathbb{D}^*$.

As $E$ is a non-flat embedded end of finite total curvature and branching order $n$, it is a catenoidal end if $n=0$ and a planar end 
if $n\geq 1$.  In either case, up to rotation and homothety, $E$ has Weierstrass data of the form  
$$(\mathbb{D}^*, J,  z^{n+1}, z^{n-1}\left(1 +z H_0(z)\right)\d z)$$
where $(\mathbb{D}^* 
, J)$ is the punctured disk and $H_0$ is a holomorphic function on $\mathbb{D}$. 
Writing $z=re^{i \theta}$, we compute that as $r\to 0$ that
$$
P=\left(-\frac{(3n+2)(n+2)}{ 8z^2} +O\left(\frac{1}{r^3}\right)\right) \d z^2=-\frac{(3n+2)(n+2)}{ 8}\frac{\d z^2}{z^2}
+O\left(\frac{1}{r^3}\right) \d z^2,
$$
$$
g=\left(\frac{1}{4r^4} +O\left(\frac{1}{r^3}\right) \right) \d z \otimes \d\bar{z}=\frac{1}{4 r^4} \left( \d r^2+r^2 \d 
\theta^2\right)+O\left(\frac{1}{r^3}\right) \d z \otimes \d\bar{z},
$$
and
$$
K_g=-16 (n+1)^2 r^{2n+4} +O(r^{2n+5}).
$$
Hence, the result follows from \eqref{Compu} with $\beta=-\frac{(3n+2)(n+2)}{4}$, $l=2n+4$ and $k=-4$.

At an umbilic point the computations of Corollary \ref{DoublePoleCor} imply that we can parameterize a neighborhood of the umbilic 
point by $\mathbb{D}$ so that as $r\to 0$
$$
P=-\left( \frac{3n^2+4n}{8} \right) \frac{\d z^2}{z^2} +O(1)\d z^2, \qquad g=\d z \otimes \d\bar{z}+O\left(r \right) \d z \otimes 
\d\bar{z}
$$
and
$$
K_g=-|a|^2 (n+1)^2 r^{2n}+O\left( r^{2n+1}\right).
$$
Hence, the result follows from \eqref{Compu} with $\beta=-\frac{(3n+4)n}{4}$, $l=2n$ and $k=0$.
\end{proof}
From Proposition~\ref{EndUmbilicProp} we obtain two corollaries. 
\begin{cor}\label{CountingCor}
 Let $\Sigma$ be a non-flat properly immersed minimal surface in $\Real^3$ of finite total Gauss curvature with genus $g$ and $e$ 
embedded ends, 
 then
 \begin{align*}
 \lim_{\alpha\to 0} \alpha ||{T}||_{\Sigma, \alpha} &=\frac{\pi}{4}\left( 8+12 g+10(e-2)+ \sum_{p\in \mathcal{E}\bigcup \mathcal{U}}  
\frac{n(p)}{n(p)+1} \right).
 \end{align*}
 Here $\mathcal{E}$ is the set of ends and $n(p)\geq 0$ for $p\in \mathcal{E}$ is the order of branching of the 
end, i.e., the order of branching of the extension of the Gauss map to $p$, while $\mathcal{U}$ is the 
set of umbilic points and $n(p)\geq 1$ is the order of the umbilic point for $p\in \mathcal{U}$.
\end{cor}
\begin{proof}
 As $\Sigma$ has finite total curvature,  a classic result of Osserman \cite{Osserman} implies that $\Sigma$ is conformal to a compact Riemann surface, $M$, with a finite number of punctures and that the Gauss map extends meromorphically to $M$.  Let $e_1, \ldots, e_n$ denote the punctures which correspond to the ends of $\Sigma$, and let $u_1, \ldots, u_m$ denote the umbilic points.  Pick $U_1, \ldots, U_{n+m}$ disjoint open subsets of $M$  each containing either an $e_i$ or a $u_j$. We may naturally think of the $U_i$ as open subsets of $\Sigma$.
 Notice that $\Sigma_0=\Sigma\backslash \cup_{i=1}^{n+m} U_i$ is compact and contains no umbilic points and so there is a $C>0$ so that for 
all $\alpha$,  $||T||_{\Sigma_0,\alpha}\leq C.$
 Hence,
$$\lim_{\alpha\to 0} \alpha ||T||_{\Sigma,\alpha}=\sum_{i=1}^{n+m}\lim_{\alpha\to 0} \alpha||T||_{U_i,\alpha}.$$
As each $U_i$ is either an  embedded end containing no umbilic points or contains exactly one umbilic point, Proposition~\ref{EndUmbilicProp} gives that
\begin{equation}\label{SumIdent}
\lim_{\alpha\to 0} \alpha ||{T}||_{\Sigma, \alpha} =\frac{\pi}{4} \left( \sum_{p\in \mathcal{E}}  
\frac{(3n(p)+2)(n(p)+2)}{n(p)+1}+\sum_{p\in \mathcal{U}} \frac{(3n(p)+4) n(p)}{n(p)+1} \right).
\end{equation}
The Poincar\'{e}-Hopf index theorem applied to the Hopf differential $Q$ implies that
$$
4g-4=\sum_{p\in \mathcal{E}} (n(p)-2)+\sum_{p\in \mathcal{U}} n(p).
$$
The proof is concluded by applying this identity to \eqref{SumIdent}.
\end{proof}

%

\begin{cor}
 If $\Sigma$ is a non-flat properly embedded minimal surface in $\Real^3$, then 
$$ \lim_{\alpha\to 0} \alpha ||{T}||_{\Sigma,\alpha}\geq 2\pi,$$
with equality if and only if $\Sigma$ is a catenoid.
\end{cor}
\begin{proof}
If  $\lim_{\alpha\to 0} \alpha ||{T}||_{\Sigma,\alpha}=\infty$, then there is nothing to show.  If this limit is finite, then Proposition 
\ref{FTCProp} implies that $\Sigma$ has finite total curvature.  
By the strong half-space theorem~\cite{Hoffman1990a} and the classification of embedded ends, as $\Sigma$ is not plane it must have at 
least two (catenoidal) ends.  Hence, by Corollary \ref{CountingCor},
$$
\lim_{\alpha \to 0} \alpha ||T||_{\Sigma, \alpha}\geq 2\pi
$$
with equality if and only if $\Sigma$ has genus zero, no other ends and no umbilic points.  Hence, the Gauss map extends to an
unbranched cover of the sphere, and so $\Sigma$ is the catenoid.
\end{proof}
We pose the following question:
\begin{ques}
 Let $\Sigma$ be a non-flat properly embedded minimal surface in $\Real^3$ and let $C$ be the catenoid.  Is it true that for finite $\alpha$
\begin{equation*}
 ||T^{\Sigma}||_{\Sigma, \alpha}\geq ||T^{C}||_{C, \alpha}
\end{equation*}
with equality only if $\Sigma$ is a catenoid? This is true in the limit as $\alpha\to 0$ and $\alpha\to \infty$.
\end{ques}
\subsection{Compactness properties for uniform bounds on $T$}
We conclude with a compactness result for sequences of properly embedded minimal surfaces $\Sigma_i$ which admit a uniform bound on the entropy differential.  This is a standard consequence of Theorem \ref{CurvEstThm} and the removable singularities result of~\cite{MeeksIIIa}.
\begin{thm} \label{CpctThm}
Fix $\alpha>0$ and suppose that
 $\Sigma_i$ is a sequence of properly embedded minimal surfaces in an open region $\Omega\subset \Real^3$ with entropy forms 
$T^{\Sigma_i}$ satisfying
$$||T^{\Sigma_i}||_{\Sigma_i, \alpha}\leq \bar{C}<\infty.$$
 Then there is a subsequence of the $\Sigma_i$ and  a finite (possibly empty) set of points $p_1, \ldots, p_N\in \Omega$ so that:
\setcounter{thm}{1}
\begin{enumerate}
  \item \label{CpctItemOne} On each compact set $K\subset \subset \Omega\backslash \set{p_1, \ldots, p_N},$
     $$\sup_{K\cap \Sigma_i} |A|\leq C(K)<\infty;$$
\item\label{CpctItemTwo}  $\epsilon N<2\bar{C}$ where $\epsilon=\epsilon(\alpha)>0$ is given by Theorem \ref{IntroCurvEstThm}; 
\item \label{CpctItemThree} The $\Sigma_i$ converge in $\Omega\backslash \set{p_1, \ldots, p_N}$ to a smooth minimal lamination $\mathcal{L}$ of $\Omega\backslash \set{p_1, \ldots, p_N}$.  Moreover, the closure of $\overline{\mathcal{L}}$ of $\mathcal{L}$ in $\Omega$ is a smooth lamination of $\Omega$.
\end{enumerate}
\end{thm}
\begin{proof}
 We define a sequence of Radon measures, $\mu_{i,\alpha}$, on $\Omega$ by setting
\begin{equation*}
 \mu_{i,\alpha}(U)=2^{\frac{1}{2(\alpha+1)}}\int_{\Sigma_i\cap U} |\hat{T}^{\Sigma_i}_\alpha|^{\frac{1}{\alpha+1}} \vol
\end{equation*}
so
\begin{equation*}
 \mu_{i,\alpha}(\Omega)=||T^{\Sigma_i}_\alpha||_{\Sigma_i,\alpha}\leq \bar{C}<\infty.
\end{equation*}
By the standard compactness theorem for Radon measures, up to passing to a subsequence, the $\mu_{i, \alpha}$ weak* converge to a Radon 
measure $\mu$. It follows with Theorem \ref{CurvEstThm} that if for $p\in \Omega$ there is an $r>0$ so that $B_{2r}(p)\subset \Omega$ and 
$\mu(B_{2r}(p))< \frac{1}{2} \epsilon$, then there is a constant $C>0$ so that
\begin{equation*}
 \sup_{B_{r}(p)\cap \Sigma_i} |A|^2\leq \frac{C^2}{r^2}<\infty.
\end{equation*}
By standard covering arguments and the pigeonhole principle one concludes that there are at most $N$ points $p_1, \ldots, p_N\in \Omega$ 
with $N\epsilon<2 \bar{C}$ so that no such $r$ exists.  It follows that for any compact set $K\subset \Omega \backslash \set{p_1, 
\ldots, p_N}$ we have the curvature estimate:
 $$\sup_{K\cap \Sigma_i} |A|\leq C(K)<\infty.$$
This verifies Items \eqref{CpctItemOne} and \eqref{CpctItemTwo}.

To prove Item \eqref{CpctItemThree}, we note that the uniform curvature estimates of Item \eqref{CpctItemTwo} and standard compactness results -- see Appendix B of \cite{Colding2004} -- imply that, up to passing to a further subsequence, the $\Sigma_i$ converge in $\Omega \backslash \set{p_1, \ldots, p_N}$ to a smooth minimal lamination, $\mathcal{L}$, of $\Omega\backslash \set{p_1, \ldots, p_N}$.
We claim that near each $p_i$ the lamination has quadratic curvature decay.
To prove this we apply the Lebesgue decomposition theorem to $\mu$ and to $L_\Omega$, Lebesgue measure on $\Real^3$ restricted to $\Omega$. This implies
\begin{equation*}
 \mu=\mu_{reg}+\mu_{sing}.
\end{equation*}
where $\mu_{reg}$ is absolutely continuous with respect to ${L}_\Omega$ while  $\mu_{sing}\perp {L}_\Omega $. In fact, the support of $\mu_{sing}$ is $\set{p_1, \ldots, p_n}$, because $\mathcal{L}$ is a lamination of $\Omega\backslash \set{p_1, \ldots, p_N}$.  Hence, 
$$
\lim_{\rho\to 0} \mu(B_{\rho}(p_i)\backslash \set{p_i}) = \lim_{\rho\to 0} \mu_{reg}(B_{\rho}(p_i)\backslash \set{p_i}) = 0
$$
and so there is a $\delta>0$ so that $\mu(B_{2\delta}(p_i)\backslash \set{p_i})<\epsilon$.
Hence, for $p\in B_{\delta}(p_i)\backslash \set{p_i}$ we may apply Theorem \ref{CurvEstThm} to the points $q_j\in \Sigma_j$ with $q_j\to p$ and use the smooth convergence to conclude that
\begin{equation*}
 |A|^2(p)\leq \frac{4 C^2}{|p-p_i|^2}.
\end{equation*}
Theorem 1.2 of~\cite{MeeksIIIa} then implies that each $p_i$ is a removable singularity of $\mathcal{L}$ which concludes the proof of Item \eqref{CpctItemThree}.
\end{proof}

\appendix
\section{Ricci Solitons}\label{RicciFlowSec}
We remark on an interesting connection the entropy form makes between minimal surfaces and two-dimensional Ricci solitons.  Recall, a smooth one-para\-meter family of metrics $g_t$ on a fixed manifold $M$ is a \emph{Ricci flow} provided
\begin{equation*}
 \frac{\d}{\d t} g_t = -2 \Ric_{g_t}. 
\end{equation*}
This flow was introduced by Hamilton in~\cite{HamiltonRicciFlow}.
When $M$ is a surface this simplifies to
\begin{equation*}
 \frac{\d}{\d t} g_t =- 2 K_{g_t} g_t.
\end{equation*}
We say $(M,g)$ is  a \emph{Ricci soliton} provided there is a vector field $X$ on $M$ and a constant $\lambda$ so that
\begin{equation*}
 -2\Ric_{g}=L_X g-2\lambda g.
\end{equation*}
For such a $g$ the family $g_t=(1-2\lambda t) \phi_t^* g$ is a Ricci flow -- here $\phi_t$ is the flow of $X$.  When $\lambda=0$ the soliton is steady (i.e. of unchanging geometry) while when $\lambda>0$ it is shrinking and when $\lambda<0$ it is expanding. If $X=\gnabla f$ then we say $g$ is a \emph{gradient  Ricci soliton} and $f$ is a \emph{soliton potential}. For such gradient Ricci solitons
\begin{equation*}
 L_X g= 2\gnabla^2 f.
\end{equation*}
So $g$ is a gradient Ricci soliton provided
\begin{equation*}
 \gnabla^2 f+ \Ric_{g}-\lambda g=0.
\end{equation*} 
If $M$ is a surface this implies 
\begin{equation*}
 \gnabla^2 f+ K_{g} g -\lambda g=0
\end{equation*}
which is equivalent to
$$
\Delta_gf =2(\lambda-K_g) \quad \text{and}\quad
^{g}\mathring{\nabla}^2 f=0.
$$ 
Covariant differentiation of the last equation gives (in coordinates)
$$
\gnabla_j\gnabla_k\left(\partial_i f\right)=\frac{1}{2}\left(\partial_k \Delta_g f\right)g_{ij}. 
$$
Using the identity $\gnabla_j\gnabla_k\left(\partial_i f\right)-\gnabla_k\gnabla_j\left(\partial_i f\right)=R^l_{ijk}\partial_l f$, where $R^l_{ijk}$ are the components of the curvature tensor of $g$, we obtain
$$
\left(\frac{1}{2}\partial_k \Delta_g f +K \partial _k f\right)g_{ij}=\left(\frac{1}{2}\partial_j \Delta_g f +K \partial _j f\right)g_{ik}.
$$
Contracting with $g^{ij}$ implies
$$
0=K \d f+\frac{1}{2} \d \Delta_g f=K \d f-\d K. 
$$
Hence near a point $p$ where $K_g\neq 0$ we have
\begin{equation*}
\d f = \d \log |K_g|
\end{equation*}
and so
$$
\gnabla^2 \log | K_g|= \gnabla^2 f= (\lambda- K_g)g. 
$$
From this we see that
\begin{equation*}
 \Delta_g \log | K_g|  =2(\lambda-K_g) \quad \text{and} \quad \mathring{\nabla}_g^2 \log | K_g| =0. 
\end{equation*}
The converse is also true:
\begin{prop}
 Let $(M,g)$ be a (possibly open) Riemmanian surface with $K_g\neq 0$.  It is a gradient Ricci soliton if and only if 
$$\mathring{\nabla}_g^2 \log |K_g| =0\quad \text{and}\quad \Delta_g \log |K_g| =2(\lambda-K_g) $$
 for some $\lambda\in \Real$. Moreover, if $(M,g)$ is a gradient Ricci soliton, then it has soliton potential $\log |K_g|$.  The sign of $\lambda$ depends on whether the soliton is expanding, steady or shrinking.
\end{prop}
Recall that if $(M,g)$ is a Riemannian surface with $K_g<0$ and $g$ satisfies the Ricci condition \eqref{RicIdent}, then $\hat{g}=|K_g|^{3/4} g$ satisfies $K_{\hat{g}}>0$ and $\Delta_{\hat{g}} \log K_{\hat{g}}=-2K_{\hat{g}}$.  
Hence, a consequence of Theorem \ref{MainCharacterizationThm} and a straightforward computation is: 
\begin{cor}
 The metric of Enneper's surface $g_{enn}$ corresponds to the  cigar soliton metric $g_{cig}$ under the map $g\to |K_g|^{3/4} g$.  Furthermore, homotheties of $g_{enn}$ are the only minimal surface metrics which correspond to  gradient Ricci soliton metrics in this manner.
\end{cor}


\begin{thebibliography}{10}

\bibitem{MR2200000}
J.~A. Aledo, J.~A. G{\'a}lvez, and P.~Mira, \emph{Marginally trapped surfaces
  in {$\mathbb L^4$} and an extended {W}eierstrass-{B}ryant representation}, Ann.
  Global Anal. Geom. \textbf{28} (2005), no.~4, 395--415. \MR{2200000}

\bibitem{Anderson}
M.~Anderson.
\newblock{Curvature estimates for minimal surfaces in 3-manifolds}.
\newblock{\em Ann. Scient. \'{E}c. Norm. Sup.} 18:89--105, 1985.


\bibitem{jacobthomasdegree}
J.~Bernstein and T.~Mettler, \emph{The {S}chwarzian derivative and the degree
  of a classical minimal surface}, forthcoming.

\bibitem{bryant}
R.~Bryant, \emph{{Surfaces of mean curvature one in hyperbolic space}},
  Ast\'{e}rique \textbf{155} (1987), 321--347.

\bibitem{MR1656822}
D.~M.~J. Calderbank, \emph{M\"obius structures and two-dimensional
  {E}instein-{W}eyl geometry}, J. Reine Angew. Math. \textbf{504} (1998),
  37--53.

\bibitem{MR655419}
S.~S. Chern and R.~Osserman, \emph{Remarks on the {R}iemannian metric of a
  minimal submanifold}, Geometry {S}ymposium, {U}trecht 1980 ({U}trecht, 1980),
  Lecture Notes in Math., vol. 894, Springer, Berlin, 1981, pp.~49--90.

\bibitem{Chow1991}
B.~Chow, \emph{The {R}icci flow on the 2-sphere}, J. Differential Geom.
  \textbf{33} (1991), 325--334.

\bibitem{ChowBook}
B.~Chow, P.~Lu, and L.~Ni, \emph{Hamilton's {R}icci flow}, Graduate Studies in
  Mathematics, vol.~77, American Mathematical Society, Providence, RI, 2006.

\bibitem{DurenEtAl}
M.~Chuaqui, P.~Duren, and B.~Osgood, \emph{The {S}chwarzian derivative for
  harmonic mappings}, J. Anal. Math. \textbf{91} (2003), 329--351.

\bibitem{Colding2004}
T.~H. Colding and W.~P. {Minicozzi II}, \emph{{The space of embedded minimal
  surfaces of fixed genus in a 3-manifold IV; Locally simply connected}}, Ann.
  of Math. (2) \textbf{160} (2004), no.~2, 573--615.

\bibitem{Fischer-Colbrie1980}
D.~Fischer-Colbrie and R.~M. Schoen, \emph{{The structure of complete stable
  minimal surfaces in 3-manifolds of nonnegative scalar curvature}}, Comm. Pure
  Appl. Math. \textbf{33} (1980), 199--211.

\bibitem{KusnerConversation}
K.~Gro\ss{}e-Brauckmann, N.~J. Korevaar, R.~B. Kusner, and J.~M. Sullivan,
  \emph{Moduli spaces of constant mean curvature surfaces and projective
  structures}, {In Preparation}.

\bibitem{HamiltonRicciFlow}
R.~S. Hamilton, \emph{Three-manifolds with positive {R}icci curvature}, J.
  Differential Geom. \textbf{17} (1982), no.~2, 255--306.

\bibitem{Hamilton2Sphere}
\bysame, \emph{The {R}icci flow on surfaces}, Mathematics and general
  relativity ({S}anta {C}ruz, {CA}, 1986), Contemp. Math., vol.~71, Amer. Math.
  Soc., Providence, RI, 1988, pp.~237--262.

\bibitem{Hille}
E.~Hille, \emph{Ordinary differential equations in the complex domain}, Dover
  Publications Inc., Mineola, NY, 1997, Reprint of the 1976 original.

\bibitem{Hoffman1990a}
D.~A. Hoffman and W.~H. {Meeks III}, \emph{The strong halfspace theorem for
  minimal surfaces}, Invent. Math. \textbf{101} (1990), no.~1, 373--377.


\bibitem{Osserman}
R.~Osserman. 
\newblock{Global Properties of Minimal Surfaces in $E^3$ and $E^n$}
\newblock{\em Ann. of Math.} 80(2):340-364, 1964.



\bibitem{KarcherNotes}
H.~Karcher, \emph{{Construction of minimal surfaces,}}, Univ. of Tokyo, 1989.

\bibitem{Kusner}
R.~Kusner and N.~Schmitt, \emph{The spinor representation of minimal surfaces},
  arXiv:dg-ga/9512003 (1995).

\bibitem{MeeksIIIa}
W.~H. {Meeks III}, J.~P\'{e}rez, and A.~Ros, \emph{{Local Removable Singularity
  Theorems for Minimal and $H$-Laminations}}, Preprint.

\bibitem{MRDuke} 
W.~H. {Meeks III} and  H.~Rosenberg, \emph{The minimal lamination closure theorem}, Duke Math. J.
  \textbf{133} (2006), no.~3, 467--497.
  
\bibitem{Ricci}
G.~Ricci-Cubastro, \emph{{Opere}}, vol.~1, Roma: Edizioni Cremonese, 1956.


\bibitem{White}
B.~White, 
\newblock {Curvature estimates and compactness theorems in 3-manifolds for
  surfaces that are stationary for parametric elliptic functionals}.
  \newblock {\em Invent. {M}ath.}  88(2):243--256, 1987.

\end{thebibliography}
\providecommand{\bysame}{\leavevmode\hbox to3em{\hrulefill}\thinspace}
\providecommand{\MR}{\relax\ifhmode\unskip\space\fi MR }
\providecommand{\MRhref}[2]{%
  \href{http://www.ams.org/mathscinet-getitem?mr=#1}{#2}
}
\providecommand{\href}[2]{#2}

\end{document}